\documentclass[reqno,12pt]{amsart}

\usepackage{amsfonts,amsmath,amsthm,amssymb,amscd}
\usepackage[all]{xy}
\usepackage[mathscr]{eucal}
\usepackage{mathrsfs}

\newcommand{\version}{Ver.~0.0}
\newcommand{\setversion}[1]{\renewcommand{\version}{Ver.~{#1}}}
\setversion{0.0 [2015/07/22 21:04:54]}
\setversion{0.1 Tokyo [2016/08/07 14:33:52]}
\setversion{0.5 Aarhus [2016/08/21 19:30:16]}
\setversion{0.6 [2017/01/10 12:46:22]}
\setversion{0.7 [2017/01/15 14:12:49]}
\setversion{0.85 [2017/01/17 14:03:08]}
\setversion{0.9 [2017/02/01 00:28:16]}
\setversion{0.95 [2017/03/12 10:21:51]}

\title [Real Double flag varieties]
{Real double flag varieties for the symplectic group}

\dedicatory{}

\author{Kyo Nishiyama}
\address{
Department of Physics and Mathematics\\
Aoyama Gakuin University\\
Fuchinobe 5-10-1, Sagamihara 252-5258, Japan}
\email{kyo@gem.aoyama.ac.jp}

\thanks{K.~N.\  is supported by JSPS KAKENHI Grant Numbers \#{16K05070}.}

\author{Bent {\O}rsted}
\address{
Department of Mathematics, Aarhus University,
Ny Munkegade, 8000 Aarhus C, Denmark
}
\email{orsted@math.au.dk}

\thanks{}

\date{\version\quad(compiled on \today)}

\theoremstyle{plain}
\newtheorem{theorem}{Theorem}

\newtheorem{corollary}[theorem]{Corollary}
\newtheorem{lemma}[theorem]{Lemma}

\theoremstyle{definition}

\newtheorem{goal}[theorem]{\upshape Goal and Main Results}
\theoremstyle{remark}
\newtheorem{remark}[theorem]{\upshape Remark}

\numberwithin{equation}{section}
\numberwithin{theorem}{section}

\newcommand{\Z}{\mathbb{Z}}

\newcommand{\R}{\mathbb{R}}

\newcommand{\C}{\mathbb{C}}

\newcommand{\calP}{\mathcal{P}}

\newcommand{\lie}[1]{\mathfrak{#1}}

\newcommand{\Lie}{\mathop\mathrm{Lie}\nolimits{}}

\newcounter{thmenum}
\newenvironment{thmenumerate}{%
\begin{list}{$(\thethmenum)$}{%
\usecounter{thmenum}
\setlength{\labelsep}{.5em}
\setlength{\labelwidth}{-7pt}
\setlength{\topsep}{0pt}
\setlength{\partopsep}{0pt}
\setlength{\parsep}{0pt}
\setlength{\leftmargin}{3pt}
\setlength{\rightmargin}{0pt}
\setlength{\itemindent}{\leftmargin}
\setlength{\itemsep}{0pt}
}}
{\end{list}}

\newenvironment{mynote}{\par\medskip\noindent\begin{math} \blacktriangleright \end{math}{\bfseries [Memorandom]}\ }
{\begin{math} \blacktriangleleft \end{math}}

\newcommand{\vectwo}[2]{{\renewcommand{\arraystretch}{.85}\Bigl(\begin{array}{@{\,}c@{\,}}{#1}\\ {#2}\end{array}\Bigr)}}

\newcommand{\mattwo}[4]{\Bigl(\begin{array}{@{\,}c@{\;\;}c@{\,}}{#1} & {#2} \\ {#3} & {#4} \end{array}\Bigr)}
\newcommand{\mattwovbar}[4]{\Bigl(\begin{array}{@{\,}c|c@{\,}}{#1} & {#2} \\ \hline {#3} & {#4} \end{array}\Bigr)}
\newcommand{\arraytwo}[4]{\begin{array}{@{\,}c@{\;\;}c@{\,}}{#1} & {#2} \\ {#3} & {#4} \end{array}}

\newcommand{\eb}{\boldsymbol{e}}

\newlength{\lengthcup}
\newcommand{\diag}{\qopname\relax o{diag}}
\newcommand{\rank}{\qopname\relax o{rank}}

\newcommand{\trace}{\qopname\relax o{trace}}

\newcommand{\Stab}{\qopname\relax o{Stab}}
\newcommand{\Ind}{\qopname\relax o{Ind}}
\newcommand{\CinfInd}{C^{\infty}\!\!\text{-}\Ind}

\newcommand{\Gal}{\qopname\relax o{Gal}}

\newcommand{\Ad}{\qopname\relax o{Ad}}

\newcommand{\sgn}{\qopname\relax o{sgn}}
\newcommand{\sign}{\qopname\relax o{sign}}

\renewcommand{\Re}{\qopname\relax o{Re}}

\newcommand{\closure}[1]{\overline{#1}}

\newcommand{\transpose}[1]{\,{}^t{#1}\,}

\newcommand{\conjugate}[1]{\overline{\rule{0pt}{1.2ex} #1}}

\newcommand{\restrict}{\big|}
\newcommand{\partition}{\mathcal{P}}

\newcommand{\Sym}{\mathop{\mathrm{Sym}}\nolimits}

\newcommand{\calorbit}{\mathcal{O}}

\newcommand{\GL}{\mathrm{GL}}
\newcommand{\SL}{\mathrm{SL}}
\newcommand{\OO}{\mathrm{O}}
\newcommand{\U}{\mathrm{U}}

\newcommand{\Sp}{\mathrm{Sp}}

\newcommand{\Mat}{\mathrm{M}}
\newcommand{\regMat}{\mathrm{M}^{\circ}}

\newcommand{\gl}{\lie{gl}}

\newcommand{\Grass}{\qopname\relax o{Gr}}
\newcommand{\LGrass}{\qopname\relax o{LGr}}

\newcommand{\miniP}{P_{\mathrm{min}}}
\newcommand{\minip}{\lie{p}_{\mathrm{min}}}

\newcommand{\psg}[1][ ]{parabolic subgroup{#1}}

\newcommand{\spanr}[1]{\mathrm{span}_{\R}\{ {#1} \}}
\newcommand{\spanc}[1]{\mathrm{span}_{\C}\{ {#1} \}}

\newcommand{\wdot}{\mathbin{\overset{w}{\cdot}}}

\newcommand{\qformQ}[1]{\mathscr{Q}_{#1}}

\newcommand{\Pol}{\mathrm{Pol}}

\newcommand{\chiPS}{\chi_{P_S}}
\newcommand{\chiQ}{\chi_{Q}}

\newcommand{\lfa}[2]{{#1}\boldsymbol{.}{#2}}

\newcommand{\IPtoQ}{\calP}
\newcommand{\IQtoP}{\mathcal{Q}}

\newcommand{\GC}{G_{\C}}
\newcommand{\KC}{K_{\C}}
\newcommand{\LC}{L_{\C}}
\newcommand{\QC}{Q_{\C}}
\newcommand{\PSC}{P_{S,\C}}
\newcommand{\BLC}{B_{L,\C}}
\newcommand{\GR}{G_{\R}}
\newcommand{\XC}{X_{\C}}

\newcommand{\wfrac}[2]{\dfrac{\,{#1}\,}{\,{#2}\,}}
\newcommand{\wsqrt}[1]{\sqrt{{#1}\,}}

\newcommand{\norm}[1]{\|{#1}\|}

\newcommand{\HilbGnu}{\mathcal{H}_{\nu}^G}
\newcommand{\HilbLmu}{\mathcal{H}_{\mu}^L}
\newcommand{\normGnu}[1]{\norm{#1}_{G, \nu}}
\newcommand{\normLmu}[1]{\norm{#1}_{L, \mu}}

\newcommand{\Stiefelnd}{S_{n,d}}

\begin{document}

\begin{abstract} 
 In this paper we study a key example of a Hermitian
symmetric space and a natural associated double flag variety,
namely for the real symplectic group $G$ and the symmetric
subgroup $L$, the Levi part of the Siegel parabolic $P_S$. We
give a detailed treatment of the case of the maximal parabolic
subgroups $Q$ of $L$ corresponding to Grassmannians and the product variety of $G/P_S$ and $L/Q$; in particular we classify
the $L$-orbits here, and find natural explicit integral transforms
between degenerate principal series of $L$ and $G$.
\end{abstract}

\maketitle

\section*{Introduction}

The geometry of flag varieties over the complex
numbers, and in particular double flag varieties, have been
much studied in recent years
(see, e.g., \cite{Fresse.Nishiyama.2016}, \cite{Henderson.Trapa.2012}, \cite{Travkin.2009}, \cite{FGT.2009} etc.). 
In this paper we focus on a
particular case of a \emph{real} double flag variety with the purpose
of understanding in detail (1) the orbit structure under the
natural action of the smaller reductive group (2) the construction
of natural integral transforms between degenerate principal
series representations, equivariant for the same group. Even
though aspects of (1) are known from general theory 
(e.g., \cite{Kobayashi.Matsuki.2014}, \cite{Kobayashi.T.Oshima.2013} and references therein), the
cases we treat here provide new and explicit information; and for 
(2) we also find new phenomena, using the theory of
prehomogeneous vector spaces and relative invariants. In
particular the Hermitian case we study has properties complementary to other well-known cases of (2).  
For this, we refer the readers to \cite{Kobayashi.Speh.2015}, 
\cite{Moellers.Orsted.Oshima.2016}, \cite{Kobayashi.Orsted.Pevzner.2011}, \cite{CKOP.2011}, \cite{Genkai.2009}, \cite{Said.Koufany.Genkai.2014} 
among others.

Thus in this paper we study a key example of a Hermitian
symmetric space and a natural associated double flag variety,
namely for the real symplectic group $G$ and the symmetric
subgroup $L$, the Levi part of the Siegel parabolic $P_S$. We
give a detailed treatment of the case of the maximal parabolic
subgroup $Q$ of $L$ corresponding to Grassmannians and the product variety of $G/P_S$ and $L/Q$; in particular we classify
the open $L$-orbits here, and find natural explicit integral transforms
between degenerate principal series of $L$ and $G$. We realize
these representations in their natural Hilbert spaces and determine
when the integral transforms are bounded operators. As an
application we also obtain information about the occurrence of
finite-dimensional representations of $L$ in both of these
generalized principal series representations of $G$ resp. $L$.
It follows from general principles, that our integral transforms,
depending on two complex parameters in certain half-spaces, may be meromorphically continued to the whole parameter space;
and that the residues will provide kernel operators (of Schwartz
kernel type, possibly even differential operators), also
intertwining (i.e., $L$-equivariant). 
For general background on integral
operators depending meromorphically on parameters, and
for equivariant integral operators -- introduced by T.~Kobayashi
as symmetry-breaking operators -- as we study here, see 
\cite{Kobayashi.Speh.2015}, \cite{Moellers.Orsted.Oshima.2016} and \cite{Kashiwara.Kawai.1979}.  
However, 
we shall not
pursue this aspect here, and it is our future subject.

It will be clear, that the structure of our example is such that
other Hermitian groups, in particular of tube type, will be
amenable to a similar analysis; thus we contend ourselves
here to give all details for the symplectic group only.

\bigskip

Let us fix notations and explain the content of this paper more explicitly.  
So let $ G = \Sp_{2n}(\R) $ be a real symplectic group.  
We denote a symplectic vector space of dimension $ 2 n $ by $ V = \R^{2n} $ 
with a natural symplectic form defined by 
$ \langle u, v \rangle = \transpose{u} J_n v $, where 
$ J_n 
= \begin{pmatrix}
0 & - 1_n \\
1_n & 0
\end{pmatrix} $.  
Thus, our $ G $ is identified with $ \Sp(V) $.  
Let $ V^+ = \spanr{ e_1, e_2, \dots, e_n } $ spanned by the first $ n $ fundamental basis vectors, 
which is a Lagrangian subspace of $ V $.  
Similarly, we put 
$ V^- = \spanr{ e_n, e_{n + 1}, \dots , e_{2 n} } $, a complementary Lagrangian subspace to $ V^+ $,  and 
we have a complete polarization $ V = V^+ \oplus V^- $.  
The Lagrangians $ V^+ $ and $ V^- $ are dual to each other by the symplectic form, 
so that we can and often do identify 
$ V^- = (V^+)^{\ast} $.  

Let $ P_S = \Stab_G(V^+) = \{ g \in G \mid g V^+ = V^+ \} $ 
be the stabilizer of the Lagrangian subspace $ V^+ $.  
Then $ P_S $ is a maximal \psg of $ G $ with Levi decomposition $ P_S = L \ltimes N $, 
where $ L = \Stab_G(V^+) \cap \Stab_G(V^-) $, the stabilizer of the polarization, 
and $ N $ is the unipotent radical of $ P_S $.  
We call $ P_S $ a Siegel \psg[]. 
Since $ G = \Sp(V) $ acts on Lagrangian subspaces transitively, 
$ \Lambda := G/P_S $ is the collection of all Lagrangian subspaces in $ V $.  
We call this space a Lagrangian flag variety and also denote it by $ \LGrass(\R^{2n}) $.

The Levi subgroup $ L $ of $ P_S $ is explicitly given by 
\begin{equation*}
L = 
\Bigl\{ 
\begin{pmatrix}
a & 0 \\
0 & \transpose{a}^{-1} 
\end{pmatrix} \Bigm| a \in \GL_n(\R) \Bigr\} \simeq \GL_n(\R) ,
\end{equation*}
and we consider it to be $ \GL(V^+) $ 
which acts on $ V^- = (V^+)^{\ast} $ in the contragredient manner.  
The unipotent radical $ N $ of $ P_S $ is realized in the matrix form as 
\begin{equation*}
N = 
\Bigl\{ 
\begin{pmatrix}
1 & z \\
0 & 1
\end{pmatrix} \Bigm| z \in \Sym_n(\R) \Bigr\} \simeq \Sym_n(\R) 
\end{equation*}
via the exponential map.
Note that 
$ \begin{pmatrix}
a & b \\
0 & \transpose{a}^{-1}
\end{pmatrix} \in P_S $ if and only if 
$ a \transpose{b} \in \Sym_n(\R) $, which in turn equivalent to $ a^{-1} b \in \Sym_n(\R) $.

Take a maximal \psg $ Q $ in $ L = \GL(V^+) $ which stabilizes 
$ d $-dimensional isotropic space $ U \subset V^+ $.  
Then $ \Xi_d := L/Q = \Grass_d(V^+) = \Grass_d(\R^n) $ is the Grassmannian of $ d $-dimensional spaces.  
Note that, in the standard realization, 
\begin{equation*}
Q = P_{(d, n - d)}^{\GL} = 
\Bigl\{ 
\begin{pmatrix}
\alpha & \xi \\
0 & \beta
\end{pmatrix} \Bigm| 
\alpha \in \GL_d(\R), \beta \in \GL_{n -d}(\R), \xi \in \Mat_{d, n -d}(\R)
\Bigr\} .
\end{equation*}

Now, our main concern is a double flag variety 
$ X = \Lambda \times \Xi_d = G/P_S \times L/Q $ on which 
$ L = \GL_n(\R) $ acts diagonally.  
We are strongly interested in the orbit structure of $ X $ under the action of $ L $ and 
its applications to representation theory.  

\begin{goal}
We will consider the following problems.  
\begin{thmenumerate}
\item
To prove there are finitely many $ L $-orbits on the double flag variety 
$ X = \Lambda \times \Xi_d $.  
We will give a complete classification of open orbits, and recursive strategy to determine 
the whole structure of $ L $-orbits on $ X $.  
See Theorems~\ref{theorem:finiteness-L-orbits}  and  \ref{theorem:classification-open-orbits-dfv}.
\item
To construct relative invariants on each open orbits.  
We will use them to define integral transforms between degenerate principal series representations 
of $ L $ and that of $ G $.  
For this, see \S~\ref{section:intertwiners}, especially Theorems~\ref{thm:conv-integral-operator-P} and \ref{theorem:intertwiner-Q-to-P}.
\end{thmenumerate}
\end{goal}

Here we will make a short remark on the double flag varieties over the \emph{complex} number field 
(or, more correctly, over an algebraically closed fields of characteristic zero).  

Let us complexify everything which appears in the setting above, so that 
$ \GC = \Sp_{2n}(\C) $ and $ \LC \simeq \GL_n(\C) $.  
The complexifications of the parabolics are 
$ \PSC $, the stabilizer of a Lagrangian subspace in the symplectic vector space $ \C^{2n} $, 
and 
$ \QC $, the stabilizer of a $ d $-dimensional vector space in $ \C^n $.  
Then it is known that 
the double flag variety 
$ \XC = \GC/\PSC \times \LC/ \QC $ 
has finitely many $ \LC $-orbits or 
$ \# \QC \backslash \GC / \PSC < \infty $.  
In this case, one can replace the maximal parabolic $ \QC $ by a Borel subgroup $ \BLC $ of $ \LC $, 
and still there are finitely many $ \LC $ orbits in 
$ \GC/\PSC \times \LC/\BLC $ 
(see \cite{NO.2011} and \cite[Table~2]{HNOO.2013}).

Even if there are only finitely many orbits of a complex algebraic group, say $ \LC $, acting on a smooth algebraic variety, 
there is no guarantee for finiteness of orbits of real forms in general  
\footnote{
It is known that there is a canonical bijection $ L(\R)\backslash (L/H)(\R) = \ker(H^1(C; H) \to H^1(C; L)) $, where 
$ C = \Gal(\C/\R) $ and $ H^1(C; H) $ denotes the first Galois cohomology group.
See \cite[Eq.~(II.5.6)]{Borel.Ji.2006}.}.
So our problem over reals seems impossible to be deduced from 
the results over $ \C $.

On the other hand, 
in the case of the complex full flag varieties, 
there exists a famous bijection between $ \KC $ orbits and 
$ \GR $ orbits called Matsuki correspondence \cite{Matsuki.1988}.  
Both orbits are finite in number.  
In the case of double flag varieties, there is no such known correspondences.  
It might be interesting to pursue such correspondences. 

Toshiyuki Kobayashi informed us that 
the finiteness of orbits $ \# X/ L < \infty $ also follows from 
general results on visible actions \cite{Kobayashi.2005}.  
We thank him for his kind notice.

\bigskip

\noindent
\textbf{Acknowledgement.}
K.~N.\ 
thanks Arhus University for its warm hospitality 
during the visits in August 2015 and 2016.  
Most of this work has been done in those periods.

\section{Elementary properties of $ G = \Sp_{2n}(\R) $}

In this section, we will give very well known basic facts on the symplectic group 
for the sake of fixing notations.  
We define 
\begin{equation*}
G = \Sp_{2n}(\R) 
= \{ g \in \GL_{2n}(\R) \mid \transpose{g} J_n g = J_n \} 
\quad
\text{ where } 
J_n = \mattwo{0}{-1_n}{1_n}{0} .
\end{equation*}
The following lemmas are quite elementary and well known.  
We just present them because of fixing notations.

\begin{lemma}\label{lemma:g-belongs-to-Sp}
If we write 
$ g = \mattwo{a}{b}{c}{d} \in \GL_{2n}(\R) $, 
then $ g $ belongs to $ G $ if and only if 
$ \transpose{a} c , \; \transpose{b} d \in \Sym_n(\R) $ and 
$ \transpose{a} d - \transpose{c} b = 1 $.
\end{lemma}

\begin{proof}
We rewrite $ \transpose{g} J g = J $ by coordinates, and get 
\begin{align*}
\transpose{c} a - \transpose{a} c &= 0 
& 
\transpose{c} b - \transpose{a} d &= - 1 
\\
\transpose{d} a - \transpose{b} c &= 1 
&
\transpose{d} b - \transpose{b} d &= 0 
\end{align*}
which shows the lemma.
\end{proof}

\begin{lemma}\label{lemma:inverse-symplectic-matrix}
If we write 
$ g = \mattwo{a}{b}{c}{d} \in G $, 
then 
$ g^{-1} = \mattwo{\transpose{d}}{-\transpose{b}}{-\transpose{c}}{\transpose{a}} $.
\end{lemma}

\begin{proof}
Since $ \transpose{g} J g = J $, 
we get 
$ g^{-1} 
= J^{-1} \transpose{g} J 
= - J \transpose{g} J 
= \mattwo{\transpose{d}}{-\transpose{b}}{-\transpose{c}}{\transpose{a}} $.
\end{proof}

\begin{lemma}\label{lemma:conjugation-of-psg-elements}
If we write 
$ g = \mattwo{a}{b}{c}{d} \in G $  
and 
$ p = \mattwo{x}{z}{0}{y} \in P_S $, 
then 
\begin{equation}
g^{-1} p g = 
\begin{pmatrix}
\transpose{d} x a + \transpose{d} z c - \transpose{b} y c & \transpose{d} x b + \transpose{d} z d - \transpose{b} y d 
\\
- \transpose{c} x a - \transpose{c} z c + \transpose{a} y c & - \transpose{c} x b - \transpose{c} z d + \transpose{a} y d 
\end{pmatrix} .
\end{equation}
Note that, in fact, $ y = \transpose{x}^{-1} $.
\end{lemma}

\begin{proof}
Just a calculation, using Lemma~\ref{lemma:inverse-symplectic-matrix}.
\end{proof}

A maximal compact subgroup $ K $ of $ G $ is given by 
$ K = \Sp_{2n}(\R) \cap \OO(2n) $.  

\begin{lemma}\label{lemma:max-compact-subgroup}
An element 
$ g = \mattwo{a}{b}{c}{d} \in G $  
belongs to $ K $ if and only if 
\begin{align*}
&
b = - c , \; d = a , 
\\
&
\transpose{a} b \in \Sym_n(\R) , \;\; \text{ and  } \;\; 
\transpose{a} a + \transpose{b} b = 1_n 
\end{align*}
hold.  
Consequently, 
\begin{equation}
K = 
\left\{ 
\begin{pmatrix}
a & b \\
-b & a 
\end{pmatrix}  \Bigm| 
a + i b \in \U(n) 
\right\}.  
\end{equation}
\end{lemma}

\begin{proof}
$ g \in G $ belongs to $ \OO(2n) $ if and only if 
$ \transpose{g} = g^{-1} $.
From Lemma~\ref{lemma:inverse-symplectic-matrix}, we get 
$ a = d $ and $ c = - b $.  
From~\ref{lemma:g-belongs-to-Sp}, 
we get the rest two equalities.

Note that if $ a + ib \in \U(n) $ 
\begin{align*}
(a + i b)^{\ast} \, (a + i b) 
&= (\transpose{a} - i \transpose{b}) \, (a + i b ) 
\\
&= (\transpose{a} a + \transpose{b} b ) + i ( \transpose{a} b - \transpose{b} a ) = 1_n .
\end{align*}
This last formula is equivalent to the above two equalities.
\end{proof}

\section{$ L $-orbits on the Lagrangian flag variety $ \Lambda $}
\label{section:L-orbits-on-Lagrangian-Grassmannian}

Now, let us begin with the investigation of $ L $ orbits on $ \Lambda = G/ P_S $, which should be well-known.  

Let us denote the Weyl group of $ P_S $ by $ W_{P_S} $, 
which is isomorphic to $ S_n $, the symmetric group of $ n $-th order.  
In fact, it coincides with the Weyl group of $ L $.  

By Bruhat decomposition, we have 
\begin{equation}
G/P_S 
= \bigcup_{w \in W_G} P_S w P_S/P_S
= \bigsqcup_{w \in W_{P_S} \backslash W_G/ W_{P_S}} P_S w P_S/P_S, 
\end{equation}
where in the second sum $ w $ moves over the representatives of the double cosets.
The double coset space 
$ W_{P_S} \backslash W_G/ W_{P_S} \simeq S_n \backslash (S_n \ltimes (\Z/ 2 \Z)^n) / S_n $ 
has a complete system of representatives of the form  
\begin{equation*}
\{ w_k = (1, \dots, 1, -1, \dots, -1) = (1^k, (-1)^{n - k}) \mid 0 \leq k \leq n \} \subset (\Z/ 2 \Z)^n .  
\end{equation*}
We realize $ w_k $ in $ G $ as 
\begin{equation}
w_k = 
\left( 
\begin{array}{cc|cc}
 & & {-1_{n -k}} & 
\\
& {1_k} &  &  
\\ \hline
{1_{n -k}} &  &  &  
\\
  &  &  & {1_k} 
\end{array}
\right) .
\end{equation}

\begin{lemma}\label{lemma:affine-open-in-kth-Bruhat-cell}
For $ 0 \leq k \leq n $, we temporarily write $ w = w_k $.  
Then 
$ P_S w P_S/ P_S = w (w^{-1} P_S w) P_S/P_S \simeq w^{-1} P_S w /(w^{-1} P_S w \cap P_S ) $ 
contains $ N_w $ given below as an open dense subset.  
\begin{equation}
\label{eq:open-affine-space-Nw}
N_w = 
\left\{ 
\begin{pmatrix}
1_n &  0 \\
\eta & 1_n 
\end{pmatrix}  \Bigm| 
\eta = 
\begin{pmatrix}
\zeta & \xi \\
\transpose{\xi} & 0_k 
\end{pmatrix}, \;
\zeta \in \Sym_{n - k}(\R), \, \xi \in \Mat_{n - k, k}(\R) 
\right\}
\end{equation}
\end{lemma}

\begin{proof}
Take 
$ \mattwo{x}{z}{0}{y} \in P_S $ and write 
$ w = \mattwo{a}{b}{c}{d} $, 
where $ a = d = \mattwo{0}{0}{0}{1_k} , \; 
c = - b = \mattwo{1_{n - k}}{0}{0}{0} $.    
Then, using the formula in Lemma~\ref{lemma:conjugation-of-psg-elements}, 
we can calculate as 
\begin{align}
w^{-1} \mattwo{x}{z}{0}{y} w 
&= \begin{pmatrix}
\transpose{d} x a + \transpose{d} z c - \transpose{b} y c & 
\transpose{d} x b + \transpose{d} z d - \transpose{b} y d  
\\
- \transpose{c} x a - \transpose{c} z c + \transpose{a} y c & 
- \transpose{c} x b - \transpose{c} z d + \transpose{a} y d  
\end{pmatrix}
\\
&= \begin{pmatrix}
  x_{22} + z_{21} + y_{11} & - x_{21} + z_{22} + y_{12} 
\\
- x_{12} - z_{11} + y_{21} & x_{11} - z_{12} + y_{22}
\end{pmatrix}
\\
&= 
\begin{pmatrix}
\begin{array}{cc|cc}
 y_{11} &  0      &  0      &  y_{12} \\
 z_{21} &  x_{22} & -x_{21} &  z_{22} \\ \hline
-z_{11} & -x_{12} &  x_{11} & -z_{12} \\
 y_{21} &  0      &  0      &  y_{22} 
\end{array}
\end{pmatrix}
\label{eq:conjugation-wpw}
\end{align}
Let us rewrite the last formula in the form
\begin{equation*}
\mattwo{1}{0}{\eta}{1} \mattwo{\alpha}{\beta}{0}{\delta} 
= \mattwo{\alpha}{\beta}{\eta \alpha}{\eta \beta + \delta} , 
\end{equation*}
so that we get 
\begin{align}
\eta 
&= 
\begin{pmatrix}
-z_{11} & -x_{12} \\
 y_{21} &  0      
\end{pmatrix}
\begin{pmatrix}
 y_{11} &  0      \\
 z_{21} &  x_{22} 
\end{pmatrix}^{-1}
\notag
\\
&= 
\begin{pmatrix}
-z_{11} & -x_{12} \\
 y_{21} &  0      
\end{pmatrix}
\begin{pmatrix}
 y_{11}^{-1} &  0      \\
- x_{22}^{-1}  z_{21} y_{11}^{-1} &  x_{22}^{-1} 
\end{pmatrix}
\notag
\\
&=
\begin{pmatrix}
 -z_{11} y_{11}^{-1} + x_{12} x_{22}^{-1}  z_{21} y_{11}^{-1} &  - x_{12} x_{22}^{-1} \\
  y_{21} y_{11}^{-1} &  0
\end{pmatrix}, 
\label{eq:last-formula-eta}
\end{align}
provided that $ y_{11}^{-1} $ and $ x_{22}^{-1} $ exist (an open condition).  
Note that we can take $ z_{11} $ and $ z_{21} $ arbitrary, 
and also that, if we put $ x_{21} = 0 $ and $ y_{12} = 0 $, 
we can take $ x_{12} $ (which determines $ y_{21} $) arbitrary.  
This shows the last formula \eqref{eq:last-formula-eta} above exhausts 
$ \eta $ of the form in \eqref{eq:open-affine-space-Nw}.
\end{proof}

\begin{remark}
The formula \eqref{eq:last-formula-eta} actually gives a symmetric matrix.    
One can check this directly, using $ y = \transpose{x}^{-1} $.  
See also Lemma~\ref{lemma:L-action-on-affine-open-set-Nw} below.
\end{remark}

Let us consider $ L = \GL_n(\R) $ action on the $ k $-th Bruhat cell 
$ P_S w_k P_S/P_S $.  
It is just the left multiplication.  
However, if we identify it with  
$ w^{-1} P_S w/ (w^{-1} P_S w) \cap P_S $ as in Lemma~\ref{lemma:affine-open-in-kth-Bruhat-cell}, 
the action of $ a \in L $ is given by the left multiplication of $ w^{-1} a w $.  
This conjugation is explicitly given as 
%
\begin{multline}
\label{eq:waw-expressed-by-h}
w^{-1} a w 
= 
\begin{pmatrix}
\begin{array}{cc|cc}
 h'_1   &  0      &  0      &  h'_2   \\
 0      &  h_4    & -h_3    &  0      \\ \hline
 0      & -h_2    &  h_1    &  0      \\
 h'_3   &  0      &  0      &  h'_4   
\end{array}
\end{pmatrix}
\\
\text{ where } a = \mattwo{h}{0}{0}{\transpose{h}^{-1}}, \;
h = \mattwo{h_1}{h_2}{h_3}{h_4}, \;
\transpose{h}^{-1} = h' = \mattwo{h'_1}{h'_2}{h'_3}{h'_4} ,
\end{multline}
%
which can be read off from Equation~\eqref{eq:conjugation-wpw}.

\begin{lemma}\label{lemma:L-orbits-on-kth-Bruhat-cell-finiteness-representatives}
There are exactly $ \binom{n - k + 2}{2} $ of 
$ L $-orbits on the Bruhat cell $ P_S w_k P_S/P_S \; (0 \leq k \leq n) $.  
A complete representatives of $ L $-orbits is given as 
\begin{equation*}
\begin{split}
\Bigl\{ 
\mattwo{1}{z}{0}{1} w_k P_S/P_S \Bigm| 
z = \mattwo{I_{r,s}}{0}{0}{0} \in \Sym_n(\R), \, 
0 \leq r+ s \leq n - k 
\Bigr\} ,
\\
\quad
\text{ where $ I_{r,s} = \diag(1_r, - 1_s) $.}
\end{split}
\end{equation*}
\end{lemma}

\begin{proof}
For the brevity, we will write $ w = w_k $.  
Firstly, we observe that by the left multiplication of $ L $ clearly we can choose orbit representatives from the set 
\begin{equation*}
\{ \mattwo{1}{z}{0}{1} w P_S/P_S \mid z \in \Sym_n(\R) \}  .
\end{equation*}
Then, by the calculations in the proof of 
Lemma~\ref{lemma:affine-open-in-kth-Bruhat-cell} and Equation~\eqref{eq:last-formula-eta}, 
it reduces to the subset 
\begin{equation}
\biggl\{ \mattwo{1_n}{0}{\eta}{1_n} 
\biggm| 
\eta = \mattwo{\zeta}{0}{0}{0}, \; \zeta \in \Sym_{n - k}(\R) \biggr\} \subset w^{-1} P_S w / (w^{-1} P_S w) \cap P_S.
\end{equation}
Now let us consider the action of $ L $ on this set.  
Take 
$ a = \mattwo{h}{0}{0}{\transpose{h}^{-1}} \in L $, where 
$ h = \diag(h_1, 1_k) \; (h_1 \in \GL_{n - k}(\R)) $.  
Then, the action of $ a $ is the left multiplication of $ w^{-1} a w $ as explained above 
(see Equation~\eqref{eq:waw-expressed-by-h}).  
As a consequence, it brings to 
\begin{equation*}
(w^{-1} a w) \mattwo{1}{0}{\eta}{1} P_S/ P_S
=
\begin{pmatrix}
\begin{array}{cc|cc}
h_1' & 0 & & \\
0 & 1 & & \\ \hline
h_1 \zeta & 0 & h_1 & 0 \\
0 & 0 & 0 & 1
\end{array}
\end{pmatrix} P_S/P_S
=
\begin{pmatrix}
\begin{array}{cc|cc}
1 & 0 & & \\
0 & 1 & & \\ \hline
h_1 \zeta \transpose{h_1} & 0 & 1 & 0 \\
0 & 0 & 0 & 1
\end{array}
\end{pmatrix} P_S/P_S.
\end{equation*}
Now it is well known that for a suitable choice of $ h_1 \in \GL_{n - k}(\R) $, we get 
\begin{equation*}
h_1 \zeta \transpose{h_1} = \mattwo{I_{r,s}}{0}{0}{0} 
\end{equation*}
for a certain signature $ (r, s) $ with $ r + s \leq n - k $.
\end{proof}

Let us explicitly describe the $ L $-action on 
$ N_w \subset (w^{-1} P_S w) / (w^{-1} P_S w) \cap P_S $.  

\begin{lemma}\label{lemma:L-action-on-affine-open-set-Nw}
The action of $ w^{-1} a w $ in Equation~\eqref{eq:waw-expressed-by-h} on 
\begin{equation*}
\mattwo{1_n}{0_n}{\eta}{1_n} \in N_w , \quad 
\eta = \mattwo{\zeta}{\xi}{\transpose{\xi}}{0_k} \quad 
(\zeta \in \Sym_{n - k}(\R), \; \xi \in \Mat_{n - k, k}(\R)) , 
\end{equation*}
is given by 
\begin{equation*}
\eta = \mattwo{\zeta}{\xi}{\transpose{\xi}}{0} \mapsto a \wdot \eta = \mattwo{A}{B}{\transpose{B}}{0} 
\quad
\text{ where } 
\quad
\Biggl\{ 
\begin{array}{l}
A = (h_1 + B h_3) \zeta \transpose{(h_1 + B h_3)}, 
\\[1.2ex]
B = (- h_2 + h_1 \xi) (h_4 - h_3 \xi)^{-1} .
\end{array}
\end{equation*}
So the action on $ \xi $-part is linear fractional, while 
action on $ \zeta $-part is a mixture of unimodular and linear fractional action.
\end{lemma}

\begin{proof}
Take $ a \in L $ as in Equation~\eqref{eq:waw-expressed-by-h} 
and we use the formula of $ w^{-1} a w $ there.
\begin{align*}
w^{-1} a w & \, 
\begin{pmatrix}
\begin{array}{cc|cc}
1_{n - k} & 0 & & \\
0 & 1_k & & \\ \hline
\zeta & \xi & 1_{n - k} & 0 \\
\transpose{\xi} & 0 & 0 & 1_k
\end{array}
\end{pmatrix}
= 
\begin{pmatrix}
\begin{array}{cc|cc}
 h'_1   &  0      &  0      &  h'_2   \\
 0      &  h_4    & -h_3    &  0      \\ \hline
 0      & -h_2    &  h_1    &  0      \\
 h'_3   &  0      &  0      &  h'_4   
\end{array}
\end{pmatrix}
\begin{pmatrix}
\begin{array}{cc|cc}
1_{n - k} & 0 & & \\
0 & 1_k & & \\ \hline
\zeta & \xi & 1_{n - k} & 0 \\
\transpose{\xi} & 0 & 0 & 1_k
\end{array}
\end{pmatrix}
\\
&= 
\begin{pmatrix}
\begin{array}{cc|cc}
 h'_1 + h'_2 \transpose{\xi} &  0             &  0      &  h'_2   \\
-h_3 \zeta                   &  h_4 - h_3 \xi & -h_3    &  0      \\ \hline
 h_1 \zeta                   & -h_2 + h_1 \xi &  h_1    &  0      \\
 h'_3 + h'_4 \transpose{\xi} &  0             &  0      &  h'_4   
\end{array}
\end{pmatrix}
\\
&=: \mattwovbar{1}{0}{\eta}{1} \mattwovbar{\alpha}{\beta}{0}{\delta} 
= \mattwovbar{\alpha}{\beta}{\eta \alpha}{\eta \beta + \delta} .
\end{align*}
From this, we calculate 
\begin{align*}
\eta 
&= 
\begin{pmatrix}
\begin{array}{cc}
 h_1 \zeta                   & -h_2 + h_1 \xi \\
 h'_3 + h'_4 \transpose{\xi} &  0             
\end{array}
\end{pmatrix}
\begin{pmatrix}
\begin{array}{cc}
 h'_1 + h'_2 \transpose{\xi} &  0             \\
-h_3 \zeta                   &  h_4 - h_3 \xi 
\end{array}
\end{pmatrix}^{-1}
\\
&= 
\begin{pmatrix}
\begin{array}{cc}
 h_1 \zeta                   & -h_2 + h_1 \xi \\
 h'_3 + h'_4 \transpose{\xi} &  0             
\end{array}
\end{pmatrix}
\begin{pmatrix}
\begin{array}{cc}
 (h'_1 + h'_2 \transpose{\xi})^{-1}                                &  0             \\
(h_4 - h_3 \xi)^{-1} h_3 \zeta  (h'_1 + h'_2 \transpose{\xi})^{-1} &  (h_4 - h_3 \xi)^{-1} 
\end{array}
\end{pmatrix}
\\
&=: \mattwo{A}{B}{C}{0}, 
\end{align*}
where 
\begin{align*}
A &= 
h_1 \zeta (h'_1 + h'_2 \transpose{\xi})^{-1} 
+ (-h_2 + h_1 \xi) (h_4 - h_3 \xi)^{-1} h_3 \zeta  (h'_1 + h'_2 \transpose{\xi})^{-1} , 
\\
B &= (-h_2 + h_1 \xi) (h_4 - h_3 \xi)^{-1} , 
\\
C &= ( h'_3 + h'_4 \transpose{\xi})  (h'_1 + h'_2 \transpose{\xi})^{-1} .
\end{align*}
We will rewrite these formulas neatly.  

Firstly, we notice it should hold $ B = \transpose{C} $.  
Let us check it.
For this, we compare $ h \mattwo{1}{-\xi}{0}{1} $ and $ \transpose{h}^{-1} \mattwo{1}{0}{\transpose{\xi}}{1} $.  
Using notation $ h' = \transpose{h}^{-1} $, we calculate both as 
\begin{align}
\label{eq:h1-x01}
&
h \mattwo{1}{-\xi}{0}{1} 
= 
\mattwo{h_1}{h_2}{h_3}{h_4} \mattwo{1}{-\xi}{0}{1} 
= 
\begin{pmatrix}
\begin{array}{cc}
h_1 & - h_1 \xi + h_2 \\
h_3 & - h_3 \xi + h_4
\end{array}
\end{pmatrix}
\\
&
\transpose{h}^{-1} \mattwo{1}{0}{\transpose{\xi}}{1} 
= 
h' \mattwo{1}{0}{\transpose{\xi}}{1} 
= 
\mattwo{h'_1}{h'_2}{h'_3}{h'_4} \mattwo{1}{0}{\transpose{\xi}}{1} 
= 
\begin{pmatrix}
\begin{array}{cc}
h'_1 + h'_2 \transpose{\xi} & h'_2 \\
h'_3 + h'_4 \transpose{\xi} & h'_4 
\end{array}
\end{pmatrix}
\notag
\\
\intertext{
$ \therefore \;\; $ taking transpose, }
\label{eq:1x01h-1}
&
\mattwo{1}{\xi}{0}{1} h^{-1} 
=
\begin{pmatrix}
\begin{array}{cc}
\transpose{h'_1} + \xi \transpose{h'_2} & \transpose{h'_3} + \xi \transpose{h'_4} \transpose{\xi} \\
\transpose{h'_2} & \transpose{h'_4}
\end{array}
\end{pmatrix}
\end{align}
Since \eqref{eq:h1-x01} and \eqref{eq:1x01h-1} are mutually inverse, 
we get
\begin{align}
\label{eq:hxixih-11}
&
(\transpose{h'_1} + \xi \transpose{h'_2}) h_1 
+ (\transpose{h'_3} + \xi \transpose{h'_4} \transpose{\xi}) h_3 
= 1_{n - k} , 
\\
\label{eq:hxixih-12}
&
(\transpose{h'_1} + \xi \transpose{h'_2}) (- h_1 \xi + h_2)
+ (\transpose{h'_3} + \xi \transpose{h'_4} \transpose{\xi}) (- h_3 \xi + h_4) 
= 0 , 
\\
\label{eq:hxixih-21}
&
\transpose{h'_2} h_1 + 
\transpose{h'_4} h_3 = 0 , 
\\
\label{eq:hxixih-22}
&
\transpose{h'_2} (- h_1 \xi + h_2) + 
\transpose{h'_4} (- h_3 \xi + h_4) = 1_k , 
\\
\intertext{and taking transpose of Equation \eqref{eq:hxixih-11}, }
\label{eq:hxixih-11-transposed}
&
\transpose{h_1} (h'_1 + h'_2 \transpose{\xi}) + \transpose{h_3} (h'_3 + h'_4 \transpose{\xi}) = 1_{n - k} .
\end{align}
Now, we calculate 
\begin{multline*}
\transpose{C} = 
\transpose{\Bigl[ 
(h'_3 + h'_4 \transpose{\xi}) (h'_1 + h'_2 \transpose{\xi})^{-1} 
\Bigr]}
\\
= 
(\transpose{h'_1} + \xi \transpose{h'_2})^{-1} (\transpose{h'_3} + \xi \transpose{h'_4} \transpose{\xi})
=
- (- h_1 \xi + h_2) (- h_3 \xi + h_4)^{-1}
= B , 
\end{multline*}
where in the last equality we use Equation \eqref{eq:hxixih-12}.  
This also proves the formula for linear fractional action on $ \xi $.  

Secondly, we check that $ A $ is symmetric.  
\begin{align*}
A &= 
h_1 \zeta (h'_1 + h'_2 \transpose{\xi})^{-1} + B h_3 \zeta  (h'_1 + h'_2 \transpose{\xi})^{-1} 
=
(h_1  + B h_3) \, \zeta \, (h'_1 + h'_2 \transpose{\xi})^{-1} 
\\
&= 
(h_1  + B h_3) \, \zeta \Bigl[ \transpose{h_1} + \transpose{h_3} (h'_3 + h'_4 \transpose{\xi}) (h'_1 + h'_2 \transpose{\xi})^{-1} \Bigr]
\quad (\text{by Eq.~\eqref{eq:hxixih-11-transposed}})
\\
&= ( h_1 + B h_3 ) \zeta ( \transpose{h_1} + \transpose{h_3} C ) 
\\
&= (h_1 + B h_3) \zeta \transpose{(h_1 + B h_3)} .
\qquad (\because \;\; C = \transpose{B})
\end{align*}
This proves that $ A $ is symmetric and at the same time the formula of the action on $ \zeta $ in the lemma.
\end{proof}

\begin{lemma}\label{lemma:stabilizer-L-zeta}
For a representative 
\begin{equation*}
p_{(k; r, s)} := 
\begin{pmatrix}
\begin{array}{cc|cc}
1_{n - k} & 0 & & \\
0 & 1_k & & \\ \hline
\zeta & 0 & 1_{n - k} & 0 \\
0 & 0 & 0 & 1_k
\end{array}
\end{pmatrix}, 
\qquad
\zeta = \mattwo{I_{r,s}}{0}{0}{0} \in \Sym_{n - k}(\R) 
\end{equation*}
of $ L $-orbits in the $ k $-th Bruhat cell $ w^{-1} P_S w/ (w^{-1} P_S w) \cap P_S $ 
(see Lemma~\ref{lemma:L-orbits-on-kth-Bruhat-cell-finiteness-representatives}), 
the stabilizer is given by 
\begin{equation}\label{eq:stabilizer-pkrs}
\Stab_L(p_{(k; r,s)}) = 
\Bigl\{ a = \mattwo{h}{0}{0}{\transpose{h}^{-1}} \Bigm| 
h = \mattwo{h_1}{0}{h_3}{h_4} \in \GL_n(\R),  \; 
h_1 \zeta \transpose{h_1} = \zeta \Bigr\} .
\end{equation}
Thus an orbit $ \calorbit_{(k; r,s)} $ through $ p_{(k; r, s)} $ is isomorphic to 
$ \GL_n(\R) / H_{(k; r, s)} $, where 
$ H_{(k; r,s)} $ is the collection of $ h $ given in Equation~\eqref{eq:stabilizer-pkrs}.
\end{lemma}

\begin{proof}
We put $ \xi = 0 $ in Lemma~\ref{lemma:L-action-on-affine-open-set-Nw}, and assume that 
$ a \wdot \eta = \eta $.  
It gives 
$ B = - h_2 h_4^{-1} = 0 $ 
and 
$ A = h_1 \zeta \transpose{h_1} = \zeta $.  
Here, we assume $ h_4 $ is regular.  
So, under this hypothesis, we get $ h_2 = 0 $.  
Since the stabilizer is a closed subgroup, 
we must have $ h_2 = 0 $ in any case (as a matter of fact, actually $ h_4 $ must be regular).  
\end{proof}

For the later reference, we reinterpret the above lemma by Lagrangian realization.  
Recall that $ G/P_S $ is isomorphic to the set of Lagrangian subspaces in $ V $ denoted as $ \Lambda $.  
The isomorphism is explicitly given by $ G/P_S \ni g P_S \mapsto g \cdot V^+ \in \Lambda $, 
here we identify $ V^+ $ with the space $ \spanr{\eb_1, \dots, \eb_n} $ spanned by the first $ n $ fundamental vectors in $ V = \R^{2n} $.  
For $ v = \sum_{i = 1}^n c_i \eb_i \in V^+ $, we denote $ v^{(n -k)} = \sum_{i = 1}^{n - k} c_i \eb_i $ and 
$ v_{(k)} = \sum_{j = 1}^k c_{n -k + j} \eb_{n - k + j} $ 
so that 
$ v = v^{(n - k)} + v_{(k)} $.

\begin{lemma}\label{lemma:L-orbits-on-Lagrangian-Grassmannian}
With the notation introduced above, 
$ L $-orbits on the Lagrangian Grassmannian $ \Lambda \simeq G/P_S $ has a representatives of the following form.
\begin{equation*}
V^{(k; r,s)} = \Biggl\{ u = 
\begin{pmatrix}
\begin{array}{c}
- \zeta v^{(n -k)} \\
v_{(k)} \\ \hline
v^{(n - k)} \\
0
\end{array}
\end{pmatrix} \Bigg|
v \in V^+ \Biggr\} , 
\quad
\text{where $ \zeta = \mattwo{I_{r,s}}{0}{0}{0} $.}
\end{equation*}
Here $ 0 \leq k \leq n $ and $ r, s \geq 0 $ denote the signature 
which satisfy $ 0 \leq r + s \leq k $. 
\end{lemma}

\begin{proof}
By Lemma~\ref{lemma:stabilizer-L-zeta}, we know the representatives of $ L $-orbits in the $ k $-th Bruhat cell 
$ w^{-1} P_S w/ (w^{-1} P_S w) \cap P_S \;\; (w = w_k) $.  
They are denoted as $  p_{(k; r,s)} $.  
The corresponding Lagrangian subspace is obtained by 
$ w_k \, p_{(k; r,s)} \cdot V^+ $.  
If we take $ v \in V^+ $ and write it as $ v = v^{(n -k)} + v_{(k)} $ as in just before the lemma, 
then we obtain
\begin{align*}
w_k \, p_{(k; r,s)} v 
&= 
\begin{pmatrix}
\begin{array}{cc|cc}
          &     & -1_{n - k} & \\
          & 1_k & & \\ \hline
1_{n - k} &     &  &  \\
 & & & 1_k
\end{array}
\end{pmatrix}
\begin{pmatrix}
\begin{array}{cc|cc}
1_{n - k} & 0 & & \\
0 & 1_k & & \\ \hline
\zeta & 0 & 1_{n -k} & 0 \\
0 & 0 & 0 & 1_k
\end{array}
\end{pmatrix}
\begin{pmatrix}
\begin{array}{c}
v^{(n -k)} \\
v_{(k)} \\ \hline
0 \\
0
\end{array}
\end{pmatrix} 
\\
&= 
\begin{pmatrix}
\begin{array}{cc|cc}
          &     & -1_{n - k} & \\
          & 1_k & & \\ \hline
1_{n - k} &     &  &  \\
 & & & 1_k
\end{array}
\end{pmatrix}
\begin{pmatrix}
\begin{array}{c}
v^{(n -k)} \\
v_{(k)} \\ \hline
\zeta v^{(n -k)} \\
0
\end{array}
\end{pmatrix} 
= 
\begin{pmatrix}
\begin{array}{c}
-\zeta v^{(n - k)} \\
v_{(k)} \\ \hline
v^{(n - k)} \\ 
0
\end{array}
\end{pmatrix} , 
\end{align*}
which proves the lemma.
\end{proof}

\begin{theorem}\label{theorem:finiteness-L-orbits}
Let $ B_n \subset L $ be a Borel subgroup of $ L $.  
A double flag variety 
$ G/P_S \times L/B_n $ has finitely many $ L $-orbits.  
In other words, 
$ G/P_S $ has finitely many $ B_n $-orbits.  
In this sense, $ G/P_S $ is a real $ L $-spherical variety.
\end{theorem}

\begin{proof}
Firstly, we consider the open Bruhat cell, i.e., the case where $ k = 0 $ and $ w = w_0 = J_n $.
The cell is isomorphic to 
$ w_0^{-1} P_S w_0 / (w_0^{-1} P_S w_0 \cap P_S) \simeq N_{w_0} $, where 
\begin{equation}
N_{w_0} = \Bigl\{ \mattwo{1_n}{0}{z}{1_n} \Bigm| 
z \in \Sym_n(\R) \Bigr\}
\end{equation}
and the action of $ h \in \GL_n(\R) \simeq L $ is given by 
the unimodular action: $ z \mapsto h z \transpose{h} $.  
So the complete representatives of $ L $-orbits are given by 
$ \{ z_{r,s} := \diag(1_r, -1_s, 0) \mid r, s \geq 0, \; r + s \leq n \} $.  
Let $ H_{r,s} \subset L $ be the stabilizer of $ z_{r,s} $ 
(note that $ H_{r,s} $ is denote as $ H_{(0; r,s)} $ in Lemma~\ref{lemma:stabilizer-L-zeta}.  We omit $ 0 $ for brevity).  
Then an $ L $-orbit in the open Bruhat cell is isomorphic to 
$ L/H_{r,s} $.  
What we must prove is that there are only finitely many $ B_n $ orbits on $ L/H_{r,s} $.  

Direct calculations tell that 
\begin{equation}
H_{r,s} = 
\Bigl\{ 
\mattwo{\alpha}{\beta}{0}{\gamma} \in \GL_n(\R) \Bigm| 
\alpha \in \OO(r,s), \; \gamma \in \GL_{n - (r + s)}(\R) 
\Bigr\} \subset \GL_n(\R) , 
\end{equation}
where $ \OO(r,s) $ denotes the indefinite orthogonal group 
preserving a quadratic form defined by $ \diag(1_r, -1_s) $.  
Note that if $ r + s = n $, we simply get $ H_{r,s} = \OO(r,s) $, which is a symmetric subgroup in $ \GL_n(\R) $.  
It is well known that 
a minimal parabolic subgroup $ \miniP $ has finitely many orbits on  
$ G/H $, where $ G $ is a general connected reductive Lie group, and 
$ H $ its symmetric subgroup (i.e., an open subgroup of the fixed point subgroup of a non-trivial involution of $ G $).  
For this, we refer the readers to 
\cite{Wolf.1974}, \cite{Matsuki.1979}, \cite{Rossmann.1979}.
Thus, the Borel subgroup $ B_n $ has an open orbit on $ L/\OO(r,s) $ when $ r + s = n $.  
This is equivalent to say that 
$ \lie{b}'_n + \lie{o}(r,s) = \lie{l} $ 
for some choice $ \lie{b}'_n $ of a Borel subalgebra of $ \lie{l} = \Lie L $.

On the other hand, the following is known.

\begin{lemma}\label{lemma:exist-open-orbit-implies-finiteness}
Let $ G $ be a connected reductive Lie group and $ \miniP $ its minimal \psg[].
For any closed subgroup $ H $ of $ G $, 
let us consider an action of $ H $ on the flag variety $ G/\miniP $ by the left translation.  
Then the followings are equivalent.
\begin{thmenumerate}
\item
There are finitely many $ H $-orbits in $ G/\miniP $, i.e., 
we have $ \# H \backslash G/ \miniP < \infty $.
\item
There exists an open $ H $-orbit in $ G/\miniP $.
\item
There exists $ g \in G $ for which 
$ \Ad g \cdot \lie{h}  + \minip = \lie{g} $ holds.
\end{thmenumerate}
\end{lemma}

For the proof of this lemma, 
see \cite[Remark~2.5 4)]{Kobayashi.T.Oshima.2013}.  
There is a misprint there, however.  So we repeat the remark here.  
Matsuki \cite{Matsuki.1991} observed that the lemma follows if it is valid for real rank one case, 
while the real rank one case had been already established by Kimelfeld \cite{Kimelfeld.1987}.
See also \cite{Kroetz.Schlichtkrull.2013} for another proof.

Now, in the case where $ r + s < n $, 
since the upper left corner of $ H_{r,s} $ is $ \OO(r,s) $, 
we can find a Borel subgroup $ \lie{b}'_n $ in $ \lie{l} $ for which 
$ \lie{b}'_n + \lie{h}_{r,s} = \lie{l} $ holds.  
By the above lemma, $ H_{r,s} $ has finitely many orbits in $ L/B_n $ or 
$ \# B_n \backslash L/ H_{r,s} < \infty $.  

Secondly, 
let us consider the general Bruhat cell.  
Then, by Lemma~\ref{lemma:stabilizer-L-zeta}, 
we know there are finitely many $ L $-orbits and 
they are isomorphic to $ L/H_{(k;r,s)} $.  
The Lie algebra of $ H_{(k; r,s)} $ realized in $ \GL_n(\R) $ is of the following form:
\begin{equation}
\lie{h}_{(k; r,s)} = \Bigl\{ 
\begin{pmatrix}
\begin{array}{c|c}
\arraytwo{\alpha}{\beta}{0}{\gamma} & 0 \\ \hline
\delta & \eta
\end{array}
\end{pmatrix}
\Bigm| 
\alpha \in \lie{o}(r,s) 
\Bigr\} \subset \lie{gl}_n(\R),
\end{equation}
where 
$ \mattwo{\alpha}{\beta}{0}{\gamma} \in \lie{gl}_{n - k}(\R) $ and 
$ \delta \in \Mat_{k, n -k}(\R), \; \eta \in \gl_k(\R) $.  
Let us choose a Borel subalgebra 
$ \lie{b}'_{n -k} $ of $ \gl_{n - k}(\R) $ such that 
\begin{equation*}
\lie{b}'_{n -k} + \{ \mattwo{\alpha}{\beta}{0}{\gamma} \mid \alpha \in \lie{o}(r, s) \} = \gl_{n - k}(\R) ,
\end{equation*}
applying the arguments for the open Bruhat cell.  
Then we can take
\begin{equation*}
\lie{b}_n = \begin{pmatrix}
\lie{b}'_{n - k} & \ast \\
0 & \lie{b}_k
\end{pmatrix}
\end{equation*}
as a Borel subalgebra of $ \lie{l} $ which satisfies 
$ \lie{b}_n + \lie{h}_{(k; r,s)} = \lie{l} $.  
Thus Lemma~\ref{lemma:exist-open-orbit-implies-finiteness} tells that 
$ B_n $-orbits in $ L/H_{(k; r, s)} $ is finite.
\end{proof}

\begin{corollary}
For any \psg $ Q $ of $ L $, 
the double flag variety 
$ X = \Lambda \times \Xi_d = G/P_S \times L/ Q $ has finitely many $ L $-orbits, 
hence it is of finite type.
\end{corollary}

\section{Maslov index}

In \cite{Kashiwara.Schapira.1994}, 
Kashiwara and Shapira described the orbit decomposition of the diagonal action of $ G = \Sp_{2n}(\R) $ in the triple product 
$ \Lambda^3 = \Lambda \times \Lambda \times \Lambda $ of Lagrangian Grassmannians.  
They used an invariant called Maslov index to classify the orbits and concluded that 
there are only finitely many orbits, i.e., 
$ \# \Lambda^3 / G < \infty $.

Let us explain the relation of their result and ours.

Fix points $ x_{\pm} \in \Lambda $ which are corresponding to 
the Lagrangian subspaces $ V^{\pm} \subset V $.  
We consider a $ G $-stable subspace containing 
$ \{ x_+ \} \times \{ x_- \} \times \Lambda $, namely 
Put 
\begin{equation*}
Y = G \cdot \Bigl( \{ x_+ \} \times \{ x_- \} \times \Lambda \Bigr) .
\end{equation*}
Since all the orbits go through a point $ \{ x_+ \} \times \{ x_- \} \times \{ \lambda \} $ for a certain  
$ \lambda \in \Lambda $, 
$ G $-orbit decomposition of $ Y $ 
reduces to orbit decomposition of $ \Stab_G(\{ x_+ \} \times \{ x_- \}) $ in $ \Lambda = G/ P_S $.  

It is easy to see that the stabilizer $ \Stab_G(\{ x_+ \} \times \{ x_- \}) $ is exactly $ L $ 
so that $ Y/G \simeq \Lambda/L \simeq L \backslash G/ P_S $, 
on the last of which we discussed in \S~\ref{section:L-orbits-on-Lagrangian-Grassmannian}.
Since $ Y \subset \Lambda^3 $, it has finitely many $ G $-orbits due to \cite{Kashiwara.Schapira.1994}, 
hence $ \Lambda = G/P_S $ also has finitely many $ L $-orbits.  
A detailed look at \cite{Kashiwara.Schapira.1994} will also provides the classification of orbits, which 
we do not carry out here.

However, for proving the finiteness of $ B_n $-orbits, we need explicit structure of orbits 
as homogeneous spaces of $ L $.  
This is the main point of our analysis in \S~\ref{section:L-orbits-on-Lagrangian-Grassmannian}.

\section{Classification of open $ L $-orbits in the double flag variety}

Let us return back to the original situation of Grassmannians, 
i.e., 
our $ Q = P_{(d, n -d)}^{\GL} \subset L $ is 
a maximal \psg which stabilizes a $ d $-dimensional subspace in $ V^+ $.  
So the double flag variety $ X = G/P_S \times L/Q $ is isomorphic to 
the product of the Lagrangian Grassmannian $ \Lambda = \LGrass(\R^{2n}) $ and 
the Grassmannian $ \Xi_d = \Grass_d(\R^n) $ of $ d $-dimensional subspaces.

In this section, we will describe open $ L $-orbits in $ X $.  
To study $ L $-orbits in $ X = G/P_S \times L/Q $, 
we use the identification 
\begin{equation*}
X/L \simeq Q \backslash G/ P_S \simeq \Lambda/Q .
\end{equation*}
In this identification, open $ L $-orbits corresponds to open $ Q $-orbits, 
since they are of the largest dimension.  
We already know the description of $ L $-orbits on $ \Lambda = G/P_S $ from 
\S~\ref{section:L-orbits-on-Lagrangian-Grassmannian}.
Open $ Q $-orbits are necessarily contained in open $ L $-orbits, hence 
we concentrate on the open Bruhat cell $ P_S w_0 P_S/P_S \simeq N_{w_0} \simeq \Sym_n(\R) $.
$ L $ acts on $ \Sym_n(\R) $ via unimodular action: 
$ h \cdot z = h z \transpose{h} \;\; (z \in \Sym_n(\R), \; h \in \GL_n(\R) \simeq L) $.

The following lemma, Sylvester's law of inertia, 
is a special case of Lemma~\ref{lemma:L-orbits-on-kth-Bruhat-cell-finiteness-representatives}.

\begin{lemma}\label{lemma:open-L-orbits-on-symnR}
Let $ L = \GL_n(\R) $ act on $ N_{w_0} = \Sym_n(\R) $ via unimodular action.  
Then, open orbits are parametrized by the signature $ (p, q) $ with $ p + q = n $.  
A complete system of representatives are given by $ \{ I_{p,q} \mid p, q \geq 0, \; p+ q = n \} $, 
where $ I_{p,q} = \diag (1_p, - 1_q) $.
\end{lemma}

Let us denote open $ L $-orbits by 
\begin{equation}\label{eq:def-of-Omega-pq}
\begin{aligned}
\Omega(p,q) &= \{ z \in \Sym_n(\R) \mid \text{$ z $ has signature $ (p, q) $} \}
\\
&= \{ h I_{p,q} \transpose{h} \mid h \in \GL_n(\R) \} .
\end{aligned}
\end{equation}
Thus we are looking for open $ Q $-orbits in $ \Omega(p,q) $.  
Let us denote $ H = \Stab_L(I_{p,q}) $, the stabilizer of $ I_{p,q} \in \Omega(p,q) $, 
which is isomorphic to an indefinite orthogonal group $ \OO(p,q) $.  
As a consequence $ \Omega(p,q) \simeq L/H \simeq \GL_n(\R)/ \OO(p,q) $.  

Since $ \Omega(p,q) \simeq L/H $, 
\begin{equation*}
\Omega(p,q) / Q 
\simeq H \backslash L/Q \simeq \Xi_d/H , 
\end{equation*}
where $ \Xi_d = \Grass_d(\R^n) $ is the Grassmannian of $ d $-dimensional subspaces.
So our problem of seeking $ Q $-orbits in $ \Omega(p,q) $ is equivalent to 
understand $ H $-orbits in a partial flag variety $ \Xi_d $.  
Since $ H $ is a symmetric subgroup fixed by an involutive automorphism of $ L $, 
this problem is ubiquitous in representation theory of real reductive Lie groups.

Let us consider 
a $ d $-dimensional subspace $ U = \spanr{ \eb_1, \eb_2, \dots , \eb_d } \in \Grass_{d}(\R^n) $ 
which is stabilized by $ Q $.  
Take $ z \in \Omega(p,q) $, and consider a quadratic form 
$ \qformQ{z^{-1}}(v, v) = \transpose{v} z^{-1} v \; (v \in \R^n) $ associated to $ z^{-1} $, 
which also has the same signature $ (p, q) $ as that of $ z $.  
Note that the restriction of $ \qformQ{z^{-1}} $ to $ U $ can be degenerate, and   
the rank and the signature of $ \qformQ{z^{-1}} \restrict_U $ is preserved by the action of $ Q $. 
In fact, for $ u \in U $ and $ m \in Q $, we get 
\begin{align*}
\qformQ{(m \cdot z)^{-1}}(u, u) 
&= \transpose{u} (m z \transpose{m})^{-1} u 
= \transpose{u} (\transpose{m}^{-1} z^{-1} m^{-1}) u 
\\
&= \transpose{(m^{-1} u)} z^{-1} (m^{-1} u) 
= \qformQ{z^{-1}}(m^{-1} u, m^{-1} u) .
\end{align*}
Since $ m^{-1} \in Q $ preserves $ U $, 
the quadratic forms 
$ \qformQ{z^{-1}} $ and $ \qformQ{(m \cdot z)^{-1}} $ have the same rank and the signature 
when restricted to $ U $.  
So they are clearly invariants of a $ Q $-orbit in $ \Omega(p,q) $.  
Put 
\begin{equation}
\label{eq:definition-Omega-pq-st}
\Omega(p, q; s, t) 
= \{ z \in \Omega(p,q) \mid \text{$ \qformQ{z^{-1}} \restrict_U $ has signature $ (s, t) $} \}, 
\end{equation}
where $ s + t $ is the rank of $ \qformQ{z^{-1}} \restrict_U $.  
Clearly $ 0 \leq s \leq p, \,  0 \leq t \leq q $ and $ s + t \leq d $ must be satisfied.  

\begin{lemma}\label{lemma:classification-Q-orbits-Omega-pq}
$ Q $-orbits in $ \Omega(p,q) $ are exactly 
\begin{equation*}
\{ \Omega(p, q; s, t) \mid s, t \geq 0, \, t + p \geq d, \, s + q \geq d, \, s + t \leq d \} 
\end{equation*}
given in \eqref{eq:definition-Omega-pq-st}.  
The orbit $ \Omega(p, q; s, t) $ is open if and only if 
$ s + t = d = \dim U $, 
i.e., the quadratic form $ \qformQ{z^{-1}} $ is non-degenerate when restricted to $ U $.
\end{lemma}

\begin{proof}
The restriction $ \qformQ{z^{-1}} \restrict_U $ is a quadratic form, and we denote 
its signature by $ (s, t) $.  
The rank of $ \qformQ{z^{-1}} \restrict_U $ is $ s + t $ and $ k = d - (s + t) $ is the dimension of the kernel.  
Obviously, we must have $ 0 \leq s, t , k \leq d $.
Since $ \qformQ{z^{-1}} $ is non-degenerate with signature $ (p, q) $, there exist  signature constraints 
\begin{equation*}
s + k \leq p, \quad t + k \leq q.
\end{equation*}
These conditions are equivalent to the condition given in the lemma.
The signature $ (s, t) $ and hence the dimension $ k $ of the kernel is invariant under the action of $ Q $.

Conversely, if a $ d $-dimensional subspace $ U_1 $ of the quadratic space $ \R^n $ 
has the same signature $ (s, t) $ (and hence $ k $), 
it can be translated into $ U $ by the isometry group $ \OO(p,q) $ by Witt's theorem.  
This means the signature concretely classifies $ Q $-orbits.
\end{proof}

This lemma practically classifies open $ L $-orbits on $ X = \Lambda \times \Xi_d $.  
However, we rewrite it more intrinsically.

Firstly, we note that, for $ z \in \Sym_n(\R) $,  
a Lagrangian subspace $ \lambda \in \Lambda = G/ P_S $ 
in the open Bruhat cell $ P_S w_0 P_S/P_S \simeq N_{w_0} $ 
is given by 
\begin{equation*}
\lambda = \{ v = \vectwo{z x}{x} \mid x \in \R^n \} , 
\end{equation*}
and clearly such $ z $ is uniquely determined by $ \lambda $.  
We denote the Lagrangian subspace by $ \lambda_z $.  
Also, we denote a $ d $-dimensional subspace in $ \Xi_d = \Grass_d(\R^n) $ by $ \xi $.  

\begin{theorem}\label{theorem:classification-open-orbits-dfv}
Suppose that non-negative integers $ p,q $ and $ s, t $ satisfies 
\begin{equation}
p + q = n, \; s + t = d, \; 
0 \leq s \leq p, \; 0 \leq t \leq q. 
\end{equation}
Then an open $ L $-orbit in $ X = \Lambda \times \Xi_d $ is given by 
\begin{equation*}
\calorbit(p, q; s, t) 
= \{ (\lambda_z, \xi) \in \Lambda \times \Xi_d \mid \sign(z) = (p, q), \sign(\qformQ{z^{-1}} \restrict_\xi) = (s, t) \} .
\end{equation*}
Every open orbit is of this form.
\end{theorem}

\section{Relative invariants}

Let us consider the vector space 
$ \Sym_n(\R) \times \Mat_{n,d}(\R) $, on which 
$ \GL_n(\R) \times \GL_d(\R) $ acts.  
The action is given explicitly as 
\begin{equation*}
\begin{aligned}
(h, m) \cdot (z, y) &= (h z \transpose{h}, h y \transpose{m}) 
\quad
\\
&
((h, m) \in \GL_n(\R) \times \GL_d(\R), \;
(z, y) \in \Sym_n(\R) \times \Mat_{n,d}(\R) ).
\end{aligned}
\end{equation*}
Let us put $ \regMat_{n,d}(\R) := \{ y \in \Mat_{n,d}(\R) \mid \rank y = d \} $, the subset of full rank matrices in 
$ \Mat_{n,d}(\R) $.  
Then, a map $ \pi : \regMat_{n,d}(\R) \to \Xi_d = \Grass_d(\R^n) $ defined by 
$ \pi(y) := \spanr{ y_j \mid 1 \leq j \leq d } $ ($ y_j $ denotes the $ j $-th column vector of $ y $) 
is a quotient map by the action of $ \GL_d(\R) $.  Thus we get a diagram:
\begin{equation*}
\xymatrix  @R+10pt @C+10pt @M+5pt @L+3pt {
\Sym_n(\R) \times \Mat_{n,d}(\R) \ar@{<-^)}[r]^{\text{open}} & \Sym_n(\R) \times \regMat_{n,d}(\R) \ar[d]^{/\GL_d(\R)} \\
 & \Sym_n(\R) \times \Xi_d \ar@{^(->}[r]^(.6){\text{open}} & \Lambda \times \Xi_d 
}
\end{equation*}
Comparing to 
the Grassmannian, the vector space $ \Sym_n(\R) \times \Mat_{n,d}(\R) $ is easier to handle.  
In particular, we 
introduce two basic relative invariants $ \psi_1 $ and $ \psi_2 $ on $ (z, y) \in \Sym_n(\R) \times \Mat_{n,d}(\R) $ with respect to 
the above linear action, 
\begin{equation*}
\psi_1(z, y) = \det z, \qquad
\psi_2(z, y) = \det z \cdot \det (\transpose{y} z^{-1} y) .
\end{equation*}
Note that 
\begin{equation*}
\psi_2(z, y) = (-1)^d \det \begin{pmatrix}
z & y \\
\transpose{y} & 0 
\end{pmatrix}, 
\end{equation*}
so that it is actually a polynomial.  
We consider two characters of $ (h, m) \in \GL_n(\R) \times \GL_d(\R) $: 
\begin{equation}\label{eq:characters-chi1-and-chi2}
\chi_1(h,m) = (\det h)^2 , \quad
\chi_2(h, m) = (\det h)^2 (\det m)^2 .
\end{equation}
Then it is easy to check that the relative invariants $ \psi_1, \psi_2 $ 
are transformed under characters $ \chi_1^{-1}, \chi_2^{-1} $ respectively.
Let us define
\begin{align*}
\widetilde{\Omega} &= \{ (z, y) \in \Sym_n(\R) \times \Mat_{n,d}(\R) \mid 
\psi_1(z, y) \neq 0, \; \psi_2(z, y) \neq 0 \} , 
\\
\widetilde{\Omega}(p, q; s, t) &= \{ (z, y) \in \Sym_n(\R) \times \Mat_{n,d}(\R) \mid 
\sign(z) = (p, q), \;
\sign(\transpose{y} z^{-1} y) = (s, t) \} .
\end{align*}
The set $ \widetilde{\Omega} $ is clearly open and is 
a union of open $ \GL_n(\R) \times \GL_d(\R) $-orbits in 
$ \Sym_n(\R) \times \Mat_{n,d}(\R) $.  

\begin{theorem}
The sets $ \widetilde{\Omega}(p, q; s, t) $, 
where 
\begin{equation}
\label{eq:open-condition-of-pqst}
p + q = n, \;\; 
s + t = d, \;\;
0 \leq s \leq p, \;\; 0 \leq t \leq q ,
\end{equation}
are open $ \GL_n(\R) \times \GL_d(\R) $-orbits, 
and they exhaust all the open orbits in $ \Sym_n(\R) \times \Mat_{n,d}(\R) $, i.e., 
\begin{equation*}
\widetilde{\Omega} = \coprod\nolimits_{p,q,s,t} \widetilde{\Omega}(p, q; s, t) , 
\end{equation*}
where the union is taken over $ p, q, s,t $ which satisfies \eqref{eq:open-condition-of-pqst}.  
Moreover, the quotient $ \widetilde{\Omega}(p, q; s, t) /\GL_d(\R) $ is isomorphic to 
$ \Omega(p, q; s, t) $, an open $ L $-orbit in the double flag variety 
$ X = \Lambda \times \Xi_d $.
\end{theorem}

This theorem is just a paraphrase of 
Theorem~\ref{theorem:classification-open-orbits-dfv}.  

Since relative invariants are polynomials, we can consider them on 
the complexified vector space $ \Sym_n(\C) \times \Mat_{n,d}(\C) $.  
In the rest of this section, we will study them on this complexified vector space, 
and we denote it simply by $ \Sym_n \times \Mat_{n,d} $ omitting the base field.  
Similarly, we use $ \GL_n = \GL_n(\C) $, etc., for algebraic groups over $ \C $.  

Recall the characters $ \chi_1, \chi_2 $ of $ \GL_n \times \GL_d $ in \eqref{eq:characters-chi1-and-chi2}.  
The following theorem should be well-known to the experts, but 
we need the proof of it to get further results.

\begin{theorem}
$ \GL_n \times \GL_d $-module 
$ \Pol(\Sym_n \times \Mat_{n,d}) $ contains a unique non-zero relative invariant $ f(z, y) $ 
with character $ \chi_1^{-m_1} \chi_2^{-m_2} \;\; (m_1, m_2 \geq 0) $ up to non-zero scalar multiple.
This relative invariant is explicitly given by 
$ f(z, y) = (\det z)^{m_1 + m_2} (\det (\transpose{y} z^{-1} y))^{m_2} $.
\end{theorem}

\begin{proof}
In this proof, to avoid notational complexity, we consider the dual action 
\begin{equation*}
(h, m) \cdot (z, y) = (\transpose{h}^{-1} z h^{-1}, \transpose{h}^{-1} y m^{-1}) 
\quad 
((z, y) \in \Sym_n \times \Mat_{n,d}, (h, m) \in \GL_n \times \GL_d) .
\end{equation*}
To translate the results here to the original action is easy.

First, we quote results on the structure of the polynomial rings over $ \Sym_n $ and $ \Mat_{n,d} $.  
Let us denote the irreducible finite dimensional representation of $ \GL_n $ with highest weight $ \lambda $ by $ V^{(n)}(\lambda) $ 
(if $ n $ is to be well understood, we will simply write it as $ V(\lambda) $).  

\begin{lemma}\label{lemma:irr-decomp-sym-mat}
\begin{thmenumerate}
\item
As a $ \GL_n $-module, $ \Sym_n $ is multiplicity free, and the irreducible decomposition of the polynomial ring is given by 
\begin{equation}
\Pol(\Sym_n) \simeq \bigoplus_{\mu \in \partition_n} V^{(n)}(2 \mu) .
\end{equation}
\item
Assume that $ n \geq d \geq 1 $.  
As a $ \GL_n \times \GL_d $-module, $ \Mat_{n,d} $ is also multiplicity free, and the irreducible decomposition of the polynomial ring is given by 
\begin{equation}
\Pol(\Mat_{n,d}) \simeq \bigoplus_{\lambda \in \partition_d} V^{(n)}(\lambda) \otimes V^{(d)}(\lambda) .
\end{equation}
\end{thmenumerate}
\end{lemma}

Since we are looking for relative invariants for $ \GL_d $, 
it must belong to one dimensional representation space $ {\det}_d^{\ell} = V^{(d)}(\ell \varpi_d) $, 
where $ \varpi_d = (1, \dots, 1, 0, \dots, 0) $ denotes the $ d $-th fundamental weight.
Thus it must be contained in the space 
\begin{equation}
\bigl( V^{(n)}(2 \mu) \otimes V^{(n)}(\ell \varpi_d) \bigr) \otimes V^{(d)}(\ell \varpi_d) \subset \Pol(\Sym_n \times \Mat_{n,d}).  
\end{equation}
Since a relative invariant is also contained in the one dimensional representation of $ \GL_n $, say $ {\det}_n^k = V^{(n)}(k \varpi_n) $, 
$ V^{(n)}(2 \mu) \otimes V^{(n)}(\ell \varpi_d) $ must contain $ V^{(n)}(k \varpi_n) $.  
We argue 
\begin{equation*}
\begin{aligned}
V^{(n)}(2 \mu) \otimes & V^{(n)}(\ell \varpi_d) \supset V^{(n)}(k \varpi_n) 
\\
&
\iff 
2 \mu - k \varpi_n = (\ell \varpi_d)^{\ast} = \ell \varpi_{n - d} - \ell \varpi_n 
\\
&
\iff
2 \mu = (k - \ell) \varpi_n + \ell \varpi_{n - d}.
\end{aligned}
\end{equation*}
Thus, $ \ell \geq 0 $ and $ k - \ell \geq 0 $ are both even integers, 
which completely determine $ \mu $.  
So the relative invariant is unique (up to a scalar multiple) if we fix the character 
$ {\det_n}^k {\det_d}^{\ell} 
= \chi_1^{(k - \ell)/2} \chi_2^{\ell/2} $.
\end{proof}

\begin{corollary}\label{cor:finite-dim-rep-generated-by-kernel}
Let us consider the relative invariant 
\begin{equation*}
f(z, y) = (\det z)^{m_1 + m_2} (\det (\transpose{y} z^{-1} y))^{m_2} 
\qquad (m_1, \, m_2 \geq 0)
\end{equation*}
in the above theorem.  
\begin{thmenumerate}
\item
The space $ \spanc{ f(z, y) \mid z \in \Sym_n } \subset \Pol(\Mat_{n,d}) $ is stable under $ \GL_n $ and 
it is isomorphic to $ V^{(n)}(2 m_2 \varpi_d)^{\ast} \otimes V^{(d)}(2 m_2 \varpi_d)^{\ast} $ as a $ \GL_n \times \GL_d $-module.  
\item
Similarly, 
the space $ \spanc{ f(z, y) \mid y \in \Mat_{n,d} } \subset \Pol(\Sym_n) $ is stable under $ \GL_n $ and 
it is isomorphic to $ V^{(n)}(2 m_1 \varpi_n + 2 m_2 \varpi_{n - d})^{\ast} $.  
\end{thmenumerate}
\end{corollary}

\begin{proof}
It is proved that 
\begin{equation*}
f(z, y) \in 
\bigl( V^{(n)}(2 \mu) \otimes V^{(n)}(\ell \varpi_d) \bigr) \otimes V^{(d)}(\ell \varpi_d) \subset \Pol(\Sym_n \times \Mat_{n,d}) ,
\end{equation*}
where $ k = 2 m_1 + 2 m_2 , \; \ell = 2 m_2 $ and $ \mu = (k - \ell) \varpi_n + \ell \varpi_{n - d} $.
For any specialization of $ y $, this space is mapped to $ V^{(n)}(2 \mu) $ (or possibly zero), 
and if we specialize $ z $ to some symmetric matrix, it is mapped to $ V^{(n)}(\ell \varpi_d) \otimes V^{(d)}(\ell \varpi_d) $.  
This shows the results.
\end{proof}

Although, we do not need the following lemma below, 
it will be helpful to know the explicit formula for $ \det (\transpose{y} z^{-1} y) $.  
Note that we take $ z $ instead of $ z^{-1} $ in the lemma.

\begin{lemma}
Let $ [n] = \{ 1, 2, \dots, n \} $ and 
put $ \binom{[n]}{d} := \{ I \subset [n] \mid \# I = d \} $, the family of subsets in $ [n] $ of $ d $-elements.  
For $ X \in \Mat_n $ and $ I, J \in \binom{[n]}{d} $, we will denote $ X_{I, J} := ( x_{i, j} )_{i \in I, j \in J} $, 
a $ d \times d $-submatrix of $ X $.  
For $ (z, y) \in \Sym_n \times \Mat_{n,d} $, we have 
\begin{equation*}
\det (\transpose{y} z y) 
= \sum_{I, J \in \binom{[n]}{d}} \det(z_{IJ}) \, \det ((y \transpose{y})_{IJ}) 
= \sum_{I, J \in \binom{[n]}{d}} \det(z_{IJ}) \, \det (y_{I, [d]}) \, \det (y_{J, [d]}) .
\end{equation*}
\end{lemma}

We observe that $ \{ \det (y_{I, [d]}) \mid I \in \binom{[n]}{d} \} $ is the Pl\"{u}cker coordinates and 
also $ \{ \det (z_{IJ}) \mid I, J \in \binom{[n]}{d} \} $ is the coordinates for the determinantal variety of rank $ d $ 
(there are much abundance though).

\section{Degenerate principal series representations}

Let us return back to the situation over real numbers, 
and we introduce degenerate principal series for $ G = \Sp_{2n}(\R) $ and 
$ L = \GL_n(\R) $ respectively.

\subsection{Degenerate principal series for $ G/P_S $}

Let us recall $ G = \Sp_{2n}(\R) $ and its maximal \psg $ P_S $.  
Take a character $ \chiPS $ of $ P_S $, and consider a degenerate principal series representation 
\begin{equation*}
\CinfInd_{P_S}^G \chiPS := \{ f : G \to \C : C^{\infty} \mid f(g p) = \chiPS(p)^{-1} f(g) \;\; (g \in G, p \in P_S) \} , 
\end{equation*}
where $ G $ acts by left translations: 
$ \pi_{\nu}^G(g) f(x) = f(g^{-1} x) $.
In the following, we will take  
\begin{equation}\label{eq:character.chi.PS.nu}
\chiPS(p) = |\det a|^{\nu} \qquad \text{ for } \quad
p = \mattwo{a}{w}{0}{\transpose{a}^{-1}} \in P_S .
\end{equation}
(We can multiply the sign $ \sgn(\det a) $ by $ \chiPS $, if we prefer.)

Since $ \Sym_n(\R) $ is openly embedded into $ G/P_S $, 
a function $ f \in \CinfInd_{P_S}^G \chiPS $ is determined by the restriction 
$ f\restrict_{\Omega} $ where $ \Omega $ is the embedded image of $ \Sym_n(\R) $ in $ G/P_S $.
Explicitly, $ \Omega $ is defined by 
\begin{equation*}
\Omega = \left\{ J \mattwo{1}{0}{z}{1} P_S/P_S \mid z \in \Sym_n(\R) \right\} , \qquad
J = \mattwo{0}{-1_n}{1_n}{0} = w_0 , 
\end{equation*}
where $ w_0 $ is the longest element in the Weyl group, 
and we give an open embedding by 
\begin{equation}
\xymatrix @R-30pt @C+10pt @M+5pt @L+3pt {
\Sym_n(\R) \ar[r] & \Omega \ar@{^(->}[r]^{\text{open}} & G/P_S \\
z \ar@{|->}[r] & J \mattwo{1}{0}{z}{1} P_S/P_S
}
\end{equation}
In the following we mainly identify $ \Sym_n(\R) $ and $ \Omega $.  
Let us give the fractional linear action of $ G $ on $ \Sym_n(\R) $ in our setting.

\begin{lemma}
In the above identification, the linear fractional action $ \lfa{g}{z} $ 
of $ g = \mattwo{a}{b}{c}{d} \in G $ on $ z \in \Sym_n(\R) = \Omega $ is given by 
\begin{equation}
\lfa{g}{z} = - ( a z - b)(c z - d)^{-1} \in \Sym_n(\R),
\end{equation}
if $ \det (c z - d) \neq 0 $.
\end{lemma}

\begin{proof}
By the identification, $ w = \lfa{g}{z} $ corresponds to 
$ g J \mattwo{1}{0}{z}{1} P_S/P_S $.  We can calculate it as 
\begin{align*}
g J \mattwo{1}{0}{z}{1} 
&= J (J^{-1} g J) \mattwo{1}{0}{z}{1} 
= J \mattwo{d}{-c}{-b}{a} \mattwo{1}{0}{z}{1} 
\\
&= J \mattwo{d - c z}{-c}{-b + a z}{a} 
= J \mattwo{1}{0}{w}{1} \mattwo{d - c z}{-c}{0}{u}, 
\intertext{where}
w &= (a z - b)(d - c z)^{-1} \quad \text{ and } \quad u = a + w c = \transpose{(d - c z)}^{-1}.
\end{align*}
This proves the desired formula.
\end{proof}

\begin{lemma}
For $ f \in \CinfInd_{P_S}^G \chiPS $, the action of $ \pi_{\nu}^G(g) $ on $ f $ is given by 
\begin{equation*}
\pi_{\nu}^G(g) f(z) = |\det(a + z c)|^{- \nu} f(\lfa{g^{-1}}{z}) 
\qquad
(g = \mattwo{a}{b}{c}{d} \in G, \;\; z \in \Sym_n(\R)), 
\end{equation*}
where $ \chiPS(p) $ is given in \eqref{eq:character.chi.PS.nu}.
In particular, for $ h = \mattwo{a}{0}{0}{\transpose{a}^{-1}} \in L $, we get 
\begin{equation*}
\pi_{\nu}^G(h) f(z) = |\det(a)|^{- \nu} f(a^{-1} z \transpose{a}^{-1}) .
\end{equation*}
\end{lemma}

We want to discuss the completion of the $ C^{\infty} $-version of the degenerate principal series 
$ \CinfInd_{P_S}^G \chiPS $ to a representation on a Hilbert space.  
Usually, this is achieved by the compact picture, but here 
we use noncompact picture.  
To do so, we need an elementary decomposition theorem.  

Here we write 
$ P_S = L N_S , \;\; 
L = M A $, where we wrote $ N $ for $ N_S $ which is the unipotent radical of $ P_S $, 
and $ M = \SL^{\pm}_n(\R) , A = \R_+ $.  
Further, we denote $ M_K = M \cap K = \OO(n) $.
For the opposite Siegel parabolic subgroup $ \conjugate{P_S} $, 
we denote a Langlands decomposition by 
$ \conjugate{P_S} = M A \conjugate{N_S} $.

Thus we conclude $ \conjugate{N_S} M A N_S \subset K M A N_S = G $ (open embedding).
Every $ g \in G $ can be written as 
$ g = k m a n \in K M A N_S $, and we call this 
generalized Iwasawa decomposition by abuse of the terminology.
Iwasawa decomposition $ g = k m a n $ may not be unique, 
but if we require 
$ m a = \mattwo{h}{0}{0}{\transpose{h}^{-1}} $ 
for an $ h \in \Sym_n^+(\R) $, 
it is indeed unique.
This follows from the facts that the decomposition 
$ M = \OO(n) \cdot \Sym_n^+(\R) $ is unique (Cartan decomposition), 
and that $ M_K = K \cap M = \OO(n) $.

Now we describe an explicit Iwasawa decomposition 
of elements in $ \conjugate{N_S} $.

\begin{lemma}
Let $ v(z) := \mattwo{1}{0}{z}{1} \in \conjugate{N_S} \; (z \in \Sym_n(\R)) $ and 
denote $ h := \wsqrt{1_n + z^2} \in \Sym_n^+(\R) $, a positive definite symmetric matrix.  
Then we have the Iwasawa decomposition 
\begin{align*}
v(z) &= k m a n \in K M A N = G , \quad \text{ where }
\\
k &= h^{-1} \mattwo{1}{-z}{z}{1} = \mattwo{h^{-1}}{- h^{-1} z}{h^{-1}z}{h^{-1}}, 
&
h &= \wsqrt{1_n + z^2} ,
\\
m a &= \mattwo{h}{0}{0}{\transpose{h}^{-1}}, 
&
n &= \mattwo{1}{\transpose{h}^{-1} z h^{-1}}{0}{1} 
\\
a &= \alpha 1_n, \;\; 
\alpha = (\det(1 + z^2))^{\frac{1}{2n}} & & 
\end{align*}
\end{lemma}

\begin{proof}
Since 
\begin{equation*}
\mattwo{1}{z}{-z}{1} \mattwo{1}{0}{z}{1} = \mattwo{1 + z^2}{z}{0}{1}, 
\end{equation*}
we get (putting $ h = \wsqrt{1 + z^2} $)
\begin{equation*}
\wfrac{1}{\wsqrt{1 + z^2}} \mattwo{1}{z}{-z}{1} \mattwo{1}{0}{z}{1} 
= \mattwo{h}{h^{-1} z}{0}{h^{-1}}, 
= \mattwo{h}{0}{0}{\transpose{h}^{-1}} \mattwo{1}{\transpose{h}^{-1} z h^{-1}}{0}{1}.
\end{equation*}
Notice that 
$ \wfrac{1}{\wsqrt{1 + z^2}} \mattwo{1}{z}{-z}{1} $ is in $ K $, and 
its inverse is given by $ k $ in the statement of the lemma.
The rest of the statements are easy to derive.
\end{proof}

Since 
$ \conjugate{N_S} M A N_S $ is open dense in $ G $, 
$ f \in \CinfInd_{P_S}^G \chiPS $ is determined by 
$ f \restrict_{\conjugate{N_S}} $.  
We complete the space of functions on $ \conjugate{N_S} $ or $ \Sym_n(\R) $ 
by the measure $ (\det(1 + z^2))^{\nu_0 - \frac{n + 1}{2}} dz $, 
where $ \nu_0 = \Re \nu $ and $ dz $ denotes the usual Lebesgue measure, 
in order to get a Hilbert representation.
See \cite[\S~VII.1]{Knapp.redbook} for details (we use unnormalized induction, so that there is a 
shift of $ \rho_{P_S}(a) = |\det (\Ad(a)\restrict_{N_S})|^{1/2} = |\det a|^{\frac{n + 1}{2}} $).
Thus our Hilbert space is
\begin{equation}\label{eq:def-Hilbert-rep-Gnu}
\begin{aligned}
\HilbGnu &:= \{ f : \Sym_n(\R) \to \C 
\mid \normGnu{f}^2 < \infty \} , 
\quad
\text{ where }
\\
\normGnu{f}^2 &:= \int_{\Sym_n(\R)} |f(z)|^2 (\det(1 + z^2))^{\nu_0 - \frac{n + 1}{2}} dz  .
\end{aligned}
\end{equation}
We denote an induced representation 
$ \Ind_{P_S}^G \chiPS $ on the Hilbert space $ \HilbGnu $ by $ \pi_{\nu}^G $.  

\begin{remark}
The degenerate principal series 
$ \Ind_{P_S}^G \chiPS $ induced from the character 
$ \chiPS(p) = |\det a|^{\nu} $ (cf.~Eq.~\eqref{eq:character.chi.PS.nu}) 
has the unitary axis at $ \nu_0 = \frac{n + 1}{2} $.  
If $ n $ is even, there exist complementary series for 
real $ \nu $ which satisfies 
$ \frac{n}{2} < \nu < \frac{n}{2} + 1 $ 
(see \cite[Th.~4.3]{Lee.1996}).
\end{remark}

\subsection{Degenerate principal series for $ L/Q $}
\label{subsection:deg-principal-series-LQ}

In this subsection, we fix the notations for degenerate principal series 
of 
$ L = \GL_{n}(\R) $ from its maximal \psg $ Q = P_{(d, n -d)}^{\GL} $.  
We will denote 
\begin{equation}\label{eq:character-chiQ}
q = \mattwo{k}{q_{12}}{0}{k'} \in Q, \quad \text{ and } \quad 
\chiQ(q) = |\det k|^{\mu}.
\end{equation}
Then $ \chiQ $ is a character of $ Q $, 
and we consider a degenerate principal series representation 
\begin{equation*}
\CinfInd_Q^L \chiQ := \{ F : L \to \C : C^{\infty} \mid F(a q) = \chiQ(q)^{-1} F(a) \;\; (a \in L, q \in Q) \} , 
\end{equation*}
where $ L $ acts by left translations: 
$ \pi_{\mu}^L(a) F(Y) = f(a^{-1} Y) \;\; (a, Y \in L) $.  
We introduce an $L^2$-norm on this space just like usual integral over 
a maximal compact subgroup $ K_L = K \cap L = \OO(n) $:
\begin{equation}\label{eq:L2norm-integral-over-K}
\normLmu{F}^2 := 
\int_{K_L} |F(k)|^2 dk
\qquad
(F \in \CinfInd_Q^L \chiQ), 
\end{equation}
and take a completion with respect to this norm to get 
a Hilbert space $ \HilbLmu $.  
Note that the integration is in fact well-defined on $ K_L/(K \cap Q) \simeq \OO(n)/\OO(d) \times \OO(n - d) $, 
because of the right equivariance of $ F $.
Thus we get a representation $ \pi_{\mu}^L =\Ind_Q^L \chi_Q $ on the Hilbert space $ \HilbLmu $. 

To make the definition of intertwiners more easy to handle, 
we unfold the Grassmannian $ L/Q \simeq \Grass_d(\R^n) $.  
Recall 
$ \regMat_{n,d}(\R) = \{ y \in \Mat_{n,d}(\R) \mid \rank y = d \} $.  
Then, we get a map
\begin{equation}\label{eq:projection-to-Matnd}
\xymatrix @R-30pt @C+10pt @M+5pt @L+3pt {
L = \GL_n(\R) \ar[rr] & & \regMat_{n,d}(\R) \\
Y = \mattwo{y_1}{y_3}{y_2}{y_4} \ar@{|->}[rr] & & y = \vectwo{y_1}{y_2}
}
\end{equation}
which induces an isomorphism 
$ \Xi_d = L/Q \xrightarrow{\;\sim \;} \regMat_{n,d}(\R)/\GL_d(\R) $.  
Thus we can identify 
$ \CinfInd_Q^L \chi_Q $ with the space of $ C^{\infty} $ functions 
$ F : \regMat_{n,d}(\R) \to \C $ 
with the property 
$ F(yk) = |\det k|^{-\mu} F(y) $.  
In this picture, the action of $ L $ is 
just the left translation:
\begin{equation*}
\pi_{\mu}^L(a) F(y) = F(a^{-1} y)
\qquad 
(y \in \regMat_{n,d}(\R), \; a \in \GL_n(\R) = L).
\end{equation*}
To have the $L^2$-norm defined in \eqref{eq:L2norm-integral-over-K}, 
we restrict the projection map \eqref{eq:projection-to-Matnd}
from $ L = \GL_n $ to $ K_L = \OO(n) $, 
the resulting space being the Stiefel manifold of orthonormal frames
$ \Stiefelnd = \{ y \in \regMat_{n,d} \mid \transpose{y} y = 1_d \} $.  
Then $ L/Q $ is isomorphic to $ \Stiefelnd / \OO(d) $.  
The norm given in \eqref{eq:L2norm-integral-over-K} is equal to 
\begin{equation*}
\normLmu{F}^2 = \int_{\Stiefelnd} |F(v)|^2 d\sigma(v)  , 
\end{equation*}
where $ d\sigma(v) $ is the uniquely determined $ \OO(n) $-invariant non-zero measure.  
Note that $ \Stiefelnd \simeq \OO(n)/\OO(n - d) $.

\begin{remark}
The degenerate principal series 
$ \Ind_Q^L \chiQ $ induced from the character 
$ \chiQ(q) = |{\det k}|^{\mu} $ (cf.~Eq.~\eqref{eq:character-chiQ}) 
is never unitary as a representation of $ \GL_n(\R) $.  
However, if we restrict it to $ \SL_n(\R) $, 
it has the unitary axis at $ \mu_0 = \frac{n - d}{2} $.  
In addition, there exist complementary series for 
real $ \mu $ in the interval of 
$ \frac{n - d}{2} - 1 < \mu < \frac{n - d}{2} + 1 $ 
(see \cite[\S~3.5]{Howe.Lee.1999}).
\end{remark}

\begin{remark}
If you prefer the fractional linear action, 
we should make $ y_1 $ part to $ 1_d $.  
Thus we get 
\begin{equation*}
a y = \mattwo{a_1}{a_2}{a_3}{a_4} \vectwo{1}{y_2} 
= \vectwo{a_1 + a_2 y_2}{a_3 + a_4 y_2} 
= \vectwo{1}{(a_3 + a_4 y_2)(a_1 + a_2 y_2)^{-1}} ,
\end{equation*}
and the fractional linear action is given by 
\begin{equation*}
y_2 \mapsto (a_3 + a_4 y_2)(a_1 + a_2 y_2)^{-1} 
\qquad 
( y_2 \in \Mat_{n - d, d}(\R) ).
\end{equation*}
\end{remark}

\section{Intertwiners between degenerate principal series representations}\label{section:intertwiners}

In this section, we consider the following kernel function
\begin{equation}\label{eq:kernel.function.K.alpha.beta}
\begin{aligned}
K^{\alpha, \beta}(z, y) 
&
:= |\det(z)|^{\alpha} \, |\det (\transpose{y} z^{-1} y)|^{\beta} 
= |\det(z)|^{\alpha - \beta} \, \bigl|\det \mattwo{z}{y}{\transpose{y}}{0} \bigr|^{\beta} 
\\
& \qquad \qquad 
((z, y) \in \Sym_n(\R) \times \Mat_{n,d}(\R)),
\end{aligned}
\end{equation}
with complex parameters $ \alpha, \beta \in \C $.     
Using this kernel, we aim at defining two integral kernel operators $ \IPtoQ $ and $ \IQtoP $, 
which intertwine degenerate principal series representations.

\subsection{Kernel operator $ \IPtoQ $ from $ \pi_{\nu}^G $ to $ \pi_{\mu}^L $}

In this subsection, 
we define an integral kernel operator $ \IPtoQ $ 
for $ f \in \CinfInd_{P_S}^G \chiPS $ with compact support in $ \Omega(p,q) $: 
\begin{equation}\label{eq:integral-kernel-operator-IPtoQ}
\IPtoQ f(y) = \int_{\Omega(p,q)} f(z) K^{\alpha, \beta}(z, y) d\omega(z) 
\qquad 
(y \in \regMat_{n,d}(\R)) ,
\end{equation}
where $ d\omega(z) $ is an $ L $-invariant measure on the open $ L $-orbit 
$ \Omega(p,q) \subset \Omega $.  
So the operator $ \IPtoQ $ depends on the parameters $ p $ and $ q $ as well as $ \alpha $ and $ \beta $.

For 
$ h = \mattwo{a}{0}{0}{\transpose{a}^{-1}} \in L $ and 
$ f $ above, we have 
\begin{align*}
\IPtoQ (\pi_{\nu}^G(h) f) (y) 
&= \int_{\Omega(p,q)} \chiPS(a)^{-1} f(a^{-1} z \transpose{a}^{-1}) K^{\alpha, \beta}(z, y) d\omega(z) 
\\
&= \chiPS(a)^{-1} \int_{\Omega(p,q)} f(z) K^{\alpha, \beta}(a z \transpose{a}, y) d\omega(a z \transpose{a})
\\
&= \chiPS(a)^{-1} \int_{\Omega(p,q)} f(z) |\det(a)|^{2 \alpha} K^{\alpha, \beta}(z, a^{-1} y) d\omega(z) 
\\
&= |\det(a)|^{2 \alpha - \nu} \int_{\Omega(p,q)} f(z) K^{\alpha, \beta}(z, a^{-1} y) d\omega(z) 
\\
&= |\det(a)|^{2 \alpha - \nu} \pi_{\mu}^L(a) \IPtoQ f(y) .
\end{align*}
Thus, if $ \nu = 2 \alpha $, we get an intertwiner.  
In this case, we have 
$ \IPtoQ f(y k) = |\det(k)|^{2 \beta} \IPtoQ f(y)  $ so that 
\begin{equation*}
\IPtoQ f(y) \in \CinfInd_Q^L \chiQ 
\qquad 
\text{ for }
\quad
\chiQ(p) = |\det k|^{- 2\beta} \;\;
( p = \mattwo{k}{\ast}{0}{k'} ), 
\end{equation*}
if it is a $ C^{\infty} $-function on $ L/Q $.  
To get an intertwiner to $ \pi_{\mu}^L $, we should have 
$ 2 \beta = - \mu $.

As we observed 
\begin{equation*}
\Lambda = G/P_S \supset \bigcup_{p + q = n} \Omega(p,q) \qquad (\text{open}).
\end{equation*}
For each $ p, q $, 
the space 
$ \HilbLmu(p,q) := L^2(\Omega(p,q), (\det(1 + z^2))^{\nu_0 - \frac{n + 1}{2}} dz) $ 
is a closed subspace of $ \HilbGnu $ and $ L $-stable.  
From the decomposition of the base spaces, we get a direct sum decomposition 
of $ L $-modules:
\begin{equation*}
\HilbGnu = \bigoplus_{p + q = n} \HilbGnu(p, q)
\end{equation*}
Now we state one of the main theorem in this section.

\begin{theorem}\label{thm:conv-integral-operator-P}
Let $ \nu_0 := \Re \nu, \; \mu_0 := \Re \mu $ and assume that they satisfy 
inequalities 
\begin{equation}\label{thm:P-ineq-conv-radial}
n \nu_0 + d \mu_0 > \dfrac{n (n + 1)}{2} , 
\qquad
n \nu_0 - d \mu_0 > \dfrac{n (n + 1)}{2} ,
\end{equation}
and 
\begin{equation}\label{thm:P-ineq-conv-spherical}
\nu_0 + \mu_0 \geq n + 1, 
\qquad
\mu_0 \leq 0.
\end{equation}
Put $ \alpha = \nu/2, \; \beta = - \mu/2 $.  
Then the integral kernel operator $ \IPtoQ f $ 
defined in \eqref{eq:integral-kernel-operator-IPtoQ}
converges and gives a bounded linear operator 
$ \IPtoQ : \HilbGnu(p,q) \to \HilbLmu $ 
which intertwines 
$ \pi_{\nu}^G\restrict_L $ to $ \pi_{\mu}^L $.
\end{theorem}

The rest of this subsection is devoted to prove the theorem above.
Mostly we omit $ p, q $ if there is no misunderstandings and 
we write $ \nu, \mu $ instead of $ \nu_0, \mu_0 $ in the following.

Let us evaluate the square of integral $ |\IPtoQ f(y)|^2 $ point wise.  
The first evaluation is given by Cauchy-Schwartz inequality:
\begin{align*}
&|\IPtoQ f(y)|^2 
\\
&\leq 
\int_{\Omega} |f(z)|^2 (\det(1 + z^2))^{\nu - \frac{n + 1}{2}} dz 
\int_{\Omega} |K^{\alpha, \beta}(z, y)|^2 (\det(1 + z^2))^{-(\nu - \frac{n + 1}{2})} |\det z|^{-(n + 1)} dz 
\\
&\leq \normGnu{f}^2 
\int_{\Omega} |K^{\alpha, \beta}(z, y)|^2 (\det(1 + z^2))^{-(\nu - \frac{n + 1}{2})} |\det z|^{-(n + 1)} dz ,
\end{align*}
where $ dz $ is the Lebesgue measure on $ \Sym_n(\R) \simeq \R^{\frac{n (n + 1)}{2}} $,   
and we use $ d\omega(z) = |\det z|^{-\frac{n + 1}{2}} dz$.  
Since $ \alpha = \nu/2 $ and $ \beta = - \mu/2 $, the second integral becomes 
\begin{align}
\int_{\Omega} & |K^{\alpha, \beta}(z, y)|^2 (\det(1 + z^2))^{-(\nu - \frac{n + 1}{2})} |\det z|^{-(n + 1)} dz 
\notag
\\
&=
\int_{\Omega} |\det z|^{\nu + \mu - (n + 1)} 
\bigl|\det \mattwo{z}{y}{\transpose{y}}{0} \bigr|^{- \mu} 
(\det(1 + z^2))^{-(\nu - \frac{n + 1}{2})} dz .
\label{eq:square-norm-of-K-for-P}
\end{align}
To evaluate the last integral, we use polar coordinates for $ z $.  
Namely, we put $ r := \sqrt{\trace z^2} $ and write 
$ z = r \Theta $.  
Then $ \trace (\Theta^2) = 1 $, 
and 
$ \Omega^{\Theta}(p,q) = \closure{\Omega(p,q)} \cap \{ \Theta \mid \trace (\Theta^2) = 1 \} $ 
is compact.   
Using polar coordinates, we get 
$ \det z = r^n \det \Theta $ and 
\begin{equation*}
\det \mattwo{z}{y}{\transpose{y}}{0}
= \det \mattwo{r \Theta}{y}{\transpose{y}}{0}
= r^{-2d} \det \mattwo{r \Theta}{ry}{\transpose{(ry)}}{0}
= r^{n-d} \det \mattwo{\Theta}{y}{\transpose{y}}{0}
\end{equation*}
Also we note that $ dz = r^{\frac{n (n + 1)}{2} - 1} dr d\Theta $.
Thus we get 
\begin{align*}
&
\int_{\Omega(p,q)} |\det z|^{\nu + \mu - (n + 1)} 
\bigl|\det \mattwo{z}{y}{\transpose{y}}{0} \bigr|^{- \mu} 
(\det(1 + z^2))^{-(\nu - \frac{n + 1}{2})} dz 
\\
&=
\int_{\Omega^{\Theta}(p,q)} |\det \Theta|^{\nu + \mu - (n + 1)} 
\bigl|\det \mattwo{\Theta}{y}{\transpose{y}}{0} \bigr|^{- \mu} d\Theta 
\\
&\qquad\qquad \times
\int_0^{\infty}
r^{n(\nu + \mu - (n + 1)) + (n - d)(- \mu)} 
(\det(1 + r^2 \Theta^2))^{-(\nu - \frac{n + 1}{2})} r^{\frac{n (n + 1)}{2} - 1} dr
\\ 
&=
\int_{\Omega^{\Theta}(p,q)} |\det \Theta|^{\nu + \mu - (n + 1)} 
\bigl|\det \mattwo{\Theta}{y}{\transpose{y}}{0} \bigr|^{- \mu} d\Theta 
\\
&\qquad\qquad \times
\int_0^{\infty}
r^{n \nu + d \mu - \frac{n (n + 1)}{2} - 1} 
(\det(1 + r^2 \Theta^2))^{-(\nu - \frac{n + 1}{2})} dr
\end{align*}
By the assumption \eqref{thm:P-ineq-conv-spherical}, 
the integrand in the first integral over $ \Omega^{\Theta}(p,q) $ is continuous, and converges.  
For the second, we separate it according as $ r \downarrow 0 $ or $ r \uparrow \infty $.  

If $ r $ is near zero, 
the factor $ \det(1 + r^2 \Theta^2) $ is approximately $ 1 $, so the integral converges 
if $ \int_0^1 r^{n \nu + d \mu - \frac{n (n + 1)}{2} - 1} dr $ converges.
The first inequality in \eqref{thm:P-ineq-conv-radial} guarantees the convergence.

On the other hand, 
if $ r $ is large, 
the factor $ \det(1 + r^2 \Theta^2) $ is asymptotically $ r^{2 n} $, so the integral converges 
if $ \int_1^{\infty} r^{n \nu + d \mu - \frac{n (n + 1)}{2} - 1 - 2 n (\nu - \frac{n + 1}{2})} dr $ converges.
We use the second inequality in \eqref{thm:P-ineq-conv-radial} to conclude the convergence.

Thus the integral \eqref{eq:square-norm-of-K-for-P} does converge, 
and the square root of it gives a bound for the operator norm of $ \IPtoQ $.
We finished the proof of Theorem~\ref{thm:conv-integral-operator-P}.

\subsection{Kernel operator $ \IQtoP $ from $ \pi_{\mu}^L $ to $ \pi_{\nu}^G $}

Similarly, we define $ \IQtoP F(z) $, for the moment, for 
$ F(y) \in \CinfInd_Q^L \chiQ $ by 
\begin{equation}\label{eq:integral-kernel-operator-IQtoP}
\IQtoP F(z) = \int_{\Mat_{n,d}(\R)} F(y) K^{\alpha, \beta}(z, y) dy
\qquad 
(z \in \Sym_n(\R)), 
\end{equation}
where $ dy $ denotes the Lebesgue measure on $ \Mat_{n,d}(\R) $.  
We will update the definition of $ \IQtoP $ afterwards in \eqref{eq:true-def-Q}, 
although we will check $ L $-equivariance using this expression.

The integral \eqref{eq:integral-kernel-operator-IQtoP} may \emph{diverge}, but at least we can formally calculate as 
\begin{align*}
(\IQtoP \pi_{\mu}^L(a) F)(z) 
&= \int_{\Mat_{n,d}(\R)} F(a^{-1} y) K^{\alpha, \beta}(z, y) dy
\\
&= \int_{\Mat_{n,d}(\R)} F(y) K^{\alpha, \beta}(z, a y) \frac{d a y}{d y} dy 
\\
&= \int_{\Mat_{n,d}(\R)} F(y) |\det a|^{2 \alpha} K^{\alpha, \beta}(a^{-1} z \transpose{a}^{-1}, y) |\det a|^d dy 
\\
&= |\det a|^{2 \alpha + d} \chiPS(h) \pi_{\nu}^G(h) \IQtoP F(z) .
\end{align*}
Thus, if $ \chiPS(h)^{-1} = |\det a|^{2 \alpha + d} $, 
we get an intertwiner. 
Here, we need a compatibility for the action, i.e., 
$ F(y k) = |\det k|^{- \mu} F(y) \;\; (k \in \GL_d(\R)) $ and we get 
\begin{equation*}
F(y k) K^{\alpha, \beta}(z, yk) d(yk)
= |\det k|^{- \mu + 2 \beta + n} F(y) K^{\alpha, \beta}(z, y) dy.
\end{equation*}
From this we can see, if $ \mu = 2 \beta + n $, 
the integrand (or measure) $ F(y) K^{\alpha, \beta}(z, y) dy $ 
is defined over $ \regMat_{n, d}(\R)/\GL_d(\R) \simeq \OO(n)/\OO(d) {\times} \OO(n - d) $.  
This last space is compact.  
Instead of this full quotient, 
we use the Stiefel manifold $ \Stiefelnd $ introduced in \S~\ref{subsection:deg-principal-series-LQ} 
inside $ \Mat_{n, d}(\R) $.   
Thus, for $ \alpha = - (\nu + d)/2 $ and $ \beta = (\mu - n)/2 $, 
we redefine the intertwiner $ \IQtoP $ by 
\begin{equation}\label{eq:true-def-Q}
\IQtoP F(z) = \int_{\Stiefelnd} F(y) K^{\alpha, \beta}(z, y) d\sigma(y)
\qquad 
(z \in \Sym_n(\R)), 
\end{equation}
where $ d\sigma(y) $ denotes the $ \OO(n) $-invariant measure on $ \Stiefelnd $.

\begin{theorem}\label{theorem:intertwiner-Q-to-P}
Let $ \nu_0 := \Re \nu, \; \mu_0 := \Re \mu $ and assume that they satisfy 
inequalities 
\begin{equation}\label{thm-eq:Q-ineq-conv-radial}
n \nu_0 + d \mu_0 < \dfrac{n (n + 1)}{2} , 
\qquad
n \nu_0 - d \mu_0 < \dfrac{n (n + 1)}{2} ,
\end{equation}
and 
\begin{equation}\label{thm-eq:Q-ineq-conv-spherical}
\nu_0 + \mu_0 \leq n - d, 
\qquad
\mu_0 \geq n.
\end{equation}
If $ \alpha = - (\nu + d)/2 $ and $ \beta = (\mu - n)/2 $, 
the integral kernel operator $ \IQtoP $ defined in \eqref{eq:true-def-Q} 
converges and gives an $ L $-intertwiner 
$ \IQtoP : \HilbLmu \to \HilbGnu $.  
\end{theorem}

Two remarks are in order.  
First, 
the inequalities 
\eqref{thm-eq:Q-ineq-conv-radial} and 
\eqref{thm-eq:Q-ineq-conv-spherical} 
is ``opposite'' to the inequalities in Theorem~\ref{thm:conv-integral-operator-P}.  
So $ (\nu, \mu) $ does not share a common region for convergence.  
Second, the condition \eqref{thm-eq:Q-ineq-conv-spherical} in fact \emph{implies} 
\eqref{thm-eq:Q-ineq-conv-radial}.  
However, we suspect the inequality \eqref{thm-eq:Q-ineq-conv-spherical} is too strong 
to ensure the convergence.  So we leave them as they are.

\medskip

Now let us prove the theorem.  
For brevity, we denote $ \nu_0, \mu_0 $ by $ \nu, \mu $ in the following.

Since 
$ \alpha - \beta = - (\nu + d)/2 - (\mu - n)/2 = \frac{1}{2} (n - d - (\nu + \mu)) \geq 0 $ and 
$ \beta = (\mu - n)/2 \geq 0 $, 
the kernel function $ K^{\alpha,\beta}(z, y) $ is continuous.  
So the integral \eqref{eq:true-def-Q} converges.  
Let us check $ \IQtoP F(z) \in \HilbGnu $ for $ F \in \HilbLmu $.  
By Cauchy-Schwarz inequality, we get 
\begin{align*}
|\IQtoP F(z)|^2 
\leq 
\int_{\Stiefelnd} |F(y)|^2 d\sigma(y) 
\int_{\Stiefelnd} |K^{\alpha, \beta}(z, y)|^2 d\sigma(y) 
=
\normLmu{F}^2 
\int_{\Stiefelnd} |K^{\alpha, \beta}(z, y)|^2 d\sigma(y) .
\end{align*}
Thus 
\begin{align*}
\normGnu{\IQtoP F}^2 
&= \int_{\Sym_n(\R)} |\IQtoP F(z)|^2 (\det(1 + z^2))^{\nu - \frac{n + 1}{2}} dz 
\\
&\leq 
\normLmu{F}^2 
\int_{\Stiefelnd} \int_{\Sym_n(\R)} |K^{\alpha, \beta}(z, y)|^2 (\det(1 + z^2))^{\nu - \frac{n + 1}{2}} dz d\sigma(y) 
\end{align*}
Since $ \alpha = - (\nu + d)/2 $ and 
$ \beta = (\mu - n)/2 $, 
the integral of square of the kernel is 
\begin{align}
\int_{\Sym_n(\R)} & |K^{\alpha, \beta}(z, y)|^2 (\det(1 + z^2))^{\nu - \frac{n + 1}{2}} dz
\notag
\\
&=
\int_{\Sym_n(\R)} |\det z|^{-(\nu + \mu) + n - d} 
\bigl|\det \mattwo{z}{y}{\transpose{y}}{0} \bigr|^{\mu -n} 
(\det(1 + z^2))^{\nu - \frac{n + 1}{2}} dz .
\label{eq:square-norm-of-K-for-Q}
\end{align}
As in the proof of Theorem~\ref{thm:conv-integral-operator-P},
we use polar coordinate $ z = r \Theta $.  
Namely, we put $ r := \sqrt{\trace z^2} $ and write 
$ z = r \Theta $.  
If we put 
$ \Omega^{\Theta} = \{ \Theta \in \Sym_n(\R) \mid \trace (\Theta^2) = 1 \} $, 
it is compact and $ dz = r^{\frac{n (n + 1)}{2} - 1} dr d\Theta $.
Thus we get 
\begin{align*}
&
\int_{\Sym_n(\R)} 
|\det z|^{-(\nu + \mu) + n - d} 
\bigl|\det \mattwo{z}{y}{\transpose{y}}{0} \bigr|^{\mu -n} 
(\det(1 + z^2))^{\nu - \frac{n + 1}{2}} dz 
\\
&=
\int_{\Omega^{\Theta}} |\det \Theta|^{-(\nu + \mu) + n - d} 
\bigl|\det \mattwo{\Theta}{y}{\transpose{y}}{0} \bigr|^{\mu - n} d\Theta 
\\
&\qquad\qquad \times
\int_0^{\infty}
r^{n (-(\nu + \mu) + n -d) + (n - d)(\mu - n)} 
(\det(1 + r^2 \Theta^2))^{\nu - \frac{n + 1}{2}} r^{\frac{n (n + 1)}{2} - 1} dr
\\ 
&=
\int_{\Omega^{\Theta}} |\det \Theta|^{-(\nu + \mu) + n - d} 
\bigl|\det \mattwo{\Theta}{y}{\transpose{y}}{0} \bigr|^{\mu - n} d\Theta 
\\
&\qquad\qquad \times
\int_0^{\infty}
r^{- (n \nu + d \mu) + \frac{n (n + 1)}{2} - 1} 
(\det(1 + r^2 \Theta^2))^{\nu - \frac{n + 1}{2}} dr .
\end{align*}
Since the integrand in the first integral over $ \Omega^{\Theta} $ is continuous and hence converges.  
For the second, we separate it according as $ r \downarrow 0 $ or $ r \uparrow \infty $ 
as in the proof of Theorem~\ref{thm:conv-integral-operator-P}.  

When $ r $ is near zero, the integral converges 
if $ \int_0^1 r^{- (n \nu + d \mu) + \frac{n (n + 1)}{2} - 1} dr $ converges.
The convergence follows from The first inequality in \eqref{thm-eq:Q-ineq-conv-radial}.  
When $ r $ is large, 
the integral converges 
if $ \int_1^{\infty} r^{- (n \nu + d \mu) + \frac{n (n + 1)}{2} - 1 + 2 n (\nu - \frac{n + 1}{2})} dr $ converges.
We use the second inequality in \eqref{thm-eq:Q-ineq-conv-radial} for the convergence.  

This completes the proof of Theorem~\ref{theorem:intertwiner-Q-to-P}.

\subsection{Finite dimensional representations}

If $ \alpha, \beta \in \Z $, 
we can naturally consider an algebraic kernel function  
\begin{equation*}
K^{\alpha, \beta}(z, y) = \det(z)^{\alpha} \det (\transpose{y} z^{-1} y)^{\beta} 
\qquad
((z, y) \in \Sym_n(\R) \times \Mat_{n,d}(\R))
\end{equation*}
without taking absolute value.  
By abuse of notation, we use the same symbol as before.  
Similarly we also consider algebraic characters 
\begin{align*}
\chiPS(p) = \det(a)^{\nu} \qquad 
( p = \mattwo{a}{\ast}{0}{\transpose{a}^{-1}} \in P_S)
\intertext{ and } 
\chiQ(q) = \det (k)^{\mu} \qquad 
( q = \mattwo{k}{\ast}{0}{k'} \in Q ) 
\end{align*}
if $ \mu $ and $ \nu $ are integers.  
In this setting the results in the above subsections are also valid.

We make use of Corollary~\ref{cor:finite-dim-rep-generated-by-kernel} to deduce the facts on the image and kernels of integral kernel operators 
considered above.

\begin{theorem}\label{theorem:fin-dim-subrep}
For nonnegative integers $ m_1 $ and $ m_2 $, 
we put 
$ \alpha = m_1 + m_2, \beta = m_2 $ and 
define $ K^{\alpha, \beta}(z, y) $ as above.
\begin{thmenumerate}
\item\label{theorem:fin-dim-subrep:item:Q}
Put 
$ \nu = -2 (m_1 + m_2) - d $ and $ \mu = 2 m_2 + n $, 
and define the characters $ \chiPS $ and $ \chiQ $ as above.  
Then $ \Ind_Q^L \chiQ $ contains the finite dimensional representation 
$ V^{(n)}(2 m_1 \varpi_n + 2 m_2 \varpi_{n - d})^{\ast} $ as an irreducible quotient.  
On the other hand, the representation $ \Ind_{P_S}^G \chiPS $ contains 
the same finite dimensional representation of $ L $ as a subrepresentation, 
and $ \IQtoP $ intertwines these two representations.  
This subrepresentation is the same for any $ p $ and $ q $.
\item\label{theorem:fin-dim-subrep:item:P}
Assume $ 2 m_1 \geq n + 1 $ and 
put 
$ \nu = 2 (m_1 + m_2) $ and $ \mu = - 2 m_2 $.  
Define the characters $ \chiPS $ and $ \chiQ $ as above.  
Then $ \Ind_{P_S}^G \chiPS $ contains the finite dimensional representation 
$ V^{(n)}(2 m_2 \varpi_d)^{\ast} $ of $ L = \GL_n(\R) $ as an irreducible quotient.  
On the other hand $ \Ind_Q^L \chiQ $ contains the same finite dimensional representation as a subrepresentation, 
and $ \IPtoQ $ intertwines these two representations.  
The intertwiners depend on $ p $ and $ q $, so 
there are at least $ (n + 1) $ different irreducible quotients which is isomorphic to 
$ V^{(n)}(2 m_2 \varpi_d)^{\ast} $, 
while the image in $ \Ind_Q^L \chiQ $ is the same.
\end{thmenumerate}
\end{theorem}

\begin{proof}
This follows immediately from 
Corollary~\ref{cor:finite-dim-rep-generated-by-kernel} 
and Theorems~\ref{thm:conv-integral-operator-P}
and \ref{theorem:intertwiner-Q-to-P}.  
Note that $ 2 m_1 \geq n + 1 $ is required for the convergence of the integral operator.
\end{proof}

\newcommand{\FKmbf}{\boldsymbol{a}}
\newcommand{\FKm}{\alpha}

The above result illustrates how knowledge about the geometry of a double flag variety and
associated relative invariants may give information about the structure of parabolically
induced representations, and in particular about some branching laws. Let us explain, that
the branching laws in the above Theorem are consistent with other approaches to the
structure of $ \Ind_{P_S}^G \chiPS$ in Theorem~\ref{theorem:fin-dim-subrep}~\eqref{theorem:fin-dim-subrep:item:Q}.

Let us in the
following remind about the connection between this induced representation, living
on the Shilov boundary $S$ of the Hermitian symmetric space $G/K$, and the structure
of holomorphic line bundles on this symmetric space. 
Let $ \lie{g} = \lie{k} + \lie{p} $ be a Cartan decomposition, and 
$ \lie{p}_{\C} = \lie{p}^+ \oplus \lie{p}^- $ be a decomposition into irreducible representations of $ K $.  
For holomorphic polynomials on the
symmetric space we have the Schmid decomposition (see \cite{Faraut.Koranyi.1994}, XI.2.4)
of the space of polynomials
$$\mathcal P(\lie{p}^+) = \oplus_{\FKmbf} \mathcal P_{\FKmbf}(\lie{p}^+)$$
and the sum is over multi-indices ${\FKmbf}$ of integers with $\FKm_1 \geq \FKm_2 \geq \dots \geq \FKm_n \geq 0$, 
labeling (strictly speaking, here one chooses an
order so that these are the negative of)
$K$-highest weights $\FKm_1 \gamma_1 + \dots + \FKm_n \gamma_n$ with $\gamma_1, \dots, \gamma_n$
Harish-Chandra strongly orthogonal non-compact roots.    
Now by restricting polynomials to the Shilov boundary $S$ we obtain an imbedding
of the Harish-Chandra module corresponding to holomorphic sections of the
line bundle with parameter $\nu$ in the parabolically induced representation on $S$
with the same parameter. 
For concreteness, recall: For $ f \in \CinfInd_{P_S}^G \chiPS $, 
the action of $ \pi_{\nu}^G(g) $ on $ f $ is given by
\begin{equation*}
\pi_{\nu}^G(g) f(z) = |\det(a + z c)|^{- \nu} f(\lfa{g^{-1}}{z}) 
\qquad
(g = \mattwo{a}{b}{c}{d} \in G, \;\; z \in \Sym_n(\R)). 
\end{equation*}
When $\nu$ is an even integer, this is exactly the action in the (trivialized)
holomorphic bundle, now valid for holomorphic functions of $z \in \Sym_n(\C)$.
So if we can find parameters with a finite-dimensional
invariant subspace in this Harish-Chandra module, then the same module will be
an invariant subspace in $ \Ind_{P_S}^G \chiPS$.

Recall that the
maximal compact subgroup $K$ of $G$ has a complexification isomorphic to that of $L$,
and the $G/K$ is a Hermitian symmetric space of tube type. Indeed, inside
the complexified group $G_{\mathbb C}$ the two complexifications
are conjugate. Hence if we consider a finite-dimensional representation of
$G$ (or $G_{\mathbb C}$), then the branching law for each of these subgroups
will be isomorphic. 

For Hermitian symmetric spaces of tube type in general also recall
the reproducing kernel (as in \cite{Faraut.Koranyi.1994}, especially Theorem~XIII.2.4 and the
notation there) for holomorphic sections of line bundles on $G/K$,
\begin{equation}\label{eq:reproducing-kernel}
h(z,w)^{-\nu} = \sum_{\FKmbf} (\nu)_{\FKmbf} K^{\FKmbf}(z,w)
\end{equation}
and the sum is again over multi-indices ${\FKmbf}$ of integers with $\FKm_1 \geq \FKm_2 \geq \dots \geq \FKm_n \geq 0$.
Here the functions $K^{\FKmbf}(z,w)$ are (suitably normalized) reproducing kernels
of the $K$-representations $\mathcal P_{\FKmbf}(\lie{p}^+)$. 
The Pochhammer symbol is in terms of the scalar symbol in our case here
\begin{align*}
(\nu)_{\FKmbf} &= (\nu)_{\FKm_1} (\nu - 1/2)_{\FKm_2} \cdots (\nu - (n-1)/2)_{\FKm_n} 
= \prod_{i = 1}^n (\nu - (i - 1)/2)_{\FKm_i}, \quad \text{ and } \quad \\
(x)_k &= x (x + 1) \cdots (x + k - 1) = \dfrac{\Gamma(x + k)}{\Gamma(x)}.
\end{align*}
   
Recall that for positive-definiteness of the above kernel, $\nu$ must belong to the
so-called Wallach set; this means that the corresponding Harish-Chandra module is 
unitary and corresponds to a unitary reproducing-kernel representation of $G$ (or
a double covering of $G$). Here the Wallach set is 
$${\mathcal W} = \{0, \frac{1}{2}, \dots, \frac{n-1}{2} \} \cup (\frac{n-1}{2}, \infty )$$
as in \cite{Faraut.Koranyi.1994}, XIII.2.7.

On the other hand, 
if $\nu$ is a negative integer, 
the Pochhammer symbols $(\nu)_{\FKmbf}$ vanishes when $ \FKm_1 > -\nu $.  
So this gives a finite sum in the formula~\eqref{eq:reproducing-kernel} for the reproducing kernel 
corresponding to a finite-dimensional representation of $G$, and 
$\FKmbf$ labels the $K$-types occurring here as precisely those with $-\nu \geq \FKm_1$. 
By taking boundary values 
we obtain an imbedding of the $K$-finite holomorphic sections on $G/K$ to sections of the line bundle
on $G/P_S$. Recalling that for our $G$ the Harish-Chandra strongly orthogonal non-compact roots
are $2e_j$ in terms of the usual basis $e_j$, this means that the $L$-types in 
Theorem~\ref{theorem:fin-dim-subrep}~\eqref{theorem:fin-dim-subrep:item:Q} indeed occur.
Namely, we may identify the parameters by the equation   
\begin{equation*}
2 m_1 \varpi_n + 2 m_2 \varpi_{n - d} = 2(m_1 + m_2, \dots , m_1 + m_2, m_1, \dots , m_1)
\end{equation*}
with the right-hand side of the form of a multi-index $\FKmbf$ satisfying
$- \nu = 2 (m_1 + m_2) + d \geq \FKm_1$ as required above.   

Thus we have seen, that there is consistency with the results about branching laws 
from $G$ to $K$ coming
from considering finite-dimensional continuations of holomorphic discrete series
representations, and on the other hand those branching laws from $G$ to $L$
coming from our study of relative invariants and intertwining operators from
$ \Ind_{P_S}^G \chiPS $ to $ \Ind_Q^L \chiQ $.

\renewcommand{\MR}[1]{}


\def\cftil#1{\ifmmode\setbox7\hbox{$\accent"5E#1$}\else
  \setbox7\hbox{\accent"5E#1}\penalty 10000\relax\fi\raise 1\ht7
  \hbox{\lower1.15ex\hbox to 1\wd7{\hss\accent"7E\hss}}\penalty 10000
  \hskip-1\wd7\penalty 10000\box7} \def\cprime{$'$} \def\cprime{$'$}
  \def\Dbar{\leavevmode\lower.6ex\hbox to 0pt{\hskip-.23ex \accent"16\hss}D}
\providecommand{\bysame}{\leavevmode\hbox to3em{\hrulefill}\thinspace}
\providecommand{\MR}{\relax\ifhmode\unskip\space\fi MR }
\providecommand{\MRhref}[2]{%
  \href{http://www.ams.org/mathscinet-getitem?mr=#1}{#2}
}
\providecommand{\href}[2]{#2}

\end{document}